\documentclass[12pt]{article}

\usepackage[table]{xcolor}
\usepackage{amsmath, amsthm}
\usepackage{amsfonts,amssymb}
\usepackage{hyperref}
\usepackage{authblk}
\usepackage[all]{xy}
\usepackage{enumitem}
\usepackage{stmaryrd}
\usepackage[makeroom]{cancel}
\usepackage{bbm}
\usepackage{fancyvrb}

\usepackage{xstring} 

\usepackage{tikz}
\usetikzlibrary{intersections}

\newcommand{\N}{\mathbb{N}}
\newcommand{\Z}{\mathbb{Z}}
\newcommand{\Q}{\mathbb{Q}}
\newcommand{\R}{\mathbb{R}}
\newcommand{\C}{\mathbb{C}}

\newcommand{\F}{\mathbb{F}}
\newcommand{\g}{\mathfrak{g}}

\renewcommand{\l}{\mathfrak{l}}
\newcommand{\p}{\mathfrak{p}}
\renewcommand{\P}{\mathfrak{P}}
\newcommand{\PP}{\mathbb{P}}

\newcommand{\Gal}{\operatorname{Gal}}

\newcommand{\GL}{\operatorname{GL}}
\newcommand{\SL}{\operatorname{SL}}
\newcommand{\PGL}{\operatorname{PGL}}
\newcommand{\PSL}{\operatorname{PSL}}
\newcommand{\GLFl}{\GL_2(\F_\ell)}

\newcommand{\PGLFl}{\PGL_2(\F_\ell)}

\newcommand{\Fl}{{\F_\ell}}
\newcommand{\Flx}{{\F_\ell^*}}
\newcommand{\Pl}[1]{\PP^1 ( #1 )}
\newcommand{\Ker}{\operatorname{Ker}}
\renewcommand{\Im}{\operatorname{Im}}
\newcommand{\codim}{\operatorname{codim}}

\newcommand{\disc}{\operatorname{disc}}

\newcommand{\Aut}{\operatorname{Aut}}

\newcommand{\lcm}{\operatorname{lcm}}

\newcommand{\charf}{\mathbbm{1}}
\newcommand{\Frob}{\operatorname{Frob}}

\newcommand{\smat}[4]{\left[ \begin{smallmatrix} #1 & #2 \\ #3 & #4 \end{smallmatrix} \right]}
\newcommand{\mat}[4]{\left[ \begin{matrix} #1 & #2 \\ #3 & #4 \end{matrix} \right]}
\newcommand{\eps}{\varepsilon}
\newcommand{\vphi}{\varphi}
\newcommand{\mact}{\big\vert}
\newcommand{\Ei}{\mathcal{E}}
\renewcommand{\S}{\mathcal{S}}
\newcommand{\M}{\mathcal{M}}
\newcommand{\newf}{\mathcal{N}}
\newcommand{\W}{\mathcal{W}}
\newcommand{\Set}[2]{\left\{ #1 \ \Big\vert \ #2  \right\}}
\newcommand{\cond}{\operatorname{cond}}
\newcommand{\sign}{\operatorname{sign}}

\newcommand{\mfurl}[1]{\StrSubstitute{#1}{.}{/}[\mfslash]}
\newcommand{\mfref}[1]{\mfurl{#1}\href{http://www.lmfdb.org/ModularForm/GL2/Q/holomorphic/\mfslash}{#1}}

\numberwithin{equation}{subsection}

\newtheorem{thm}[equation]{Theorem}
\newtheorem{lem}[equation]{Lemma}
\newtheorem{pro}[equation]{Proposition}

\theoremstyle{definition}
\newtheorem{de}[equation]{Definition}
\newtheorem{rk}[equation]{Remark}

\newtheorem{ex}[equation]{Example}

\makeatletter
\newcommand{\subjclass}[2][2010]{%
  \let\@oldtitle\@title%
  \gdef\@title{\@oldtitle\footnotetext{#1 \emph{Mathematics subject classification:} #2}}%
}
\newcommand{\keywords}[1]{%
  \let\@@oldtitle\@title%
  \gdef\@title{\@@oldtitle\footnotetext{\emph{Key words and phrases.} #1.}}%
}
\let\c@table\c@equation
\let\c@figure\c@equation
\makeatother
\makeatother

\title{Companion forms and explicit computation of $\PGL_2$ number fields with very little ramification}
\author{Nicolas Mascot\thanks{\href{mailto:n.a.v.mascot@warwick.ac.uk}{n.a.v.mascot@warwick.ac.uk}}}
\affil{\scriptsize{University of Warwick, Coventry CV4 7AL, UK.}}

\VerbatimFootnotes

\begin{document}

\maketitle

\begin{abstract}
In previous works \cite{algo} and \cite{certif}, we described algorithms to compute the number field cut out by the mod $\ell$ representation attached to a modular form of level $N=1$. In this article, we explain how these algorithms can be generalised to forms of higher level $N$.

As an application, we compute the Galois representations attached to a few forms which are supersingular or admit a companion mod $\ell$ with $\ell=13$ (and soon $\ell=41$), and we obtain previously unknown number fields of degree $\ell+1$ whose Galois closure has Galois group $\PGLFl$  and a root discriminant that is so small that it beats records for such number fields.

Finally, we give a formula to predict the discriminant of the fields obtained by this method, and we use it to find other interesting examples, which are unfortunately out of our computational reach.
\end{abstract}

\renewcommand{\abstractname}{Acknowledgements}
\begin{abstract}
The author thanks David Roberts for his suggestion to compute Galois representations attached to forms admitting a companion mod $\ell$, and for providing the author with a few examples of such forms that were especially amenable to computation, and Ariel Pacetti and Aurel Page for stimulating discussions about the arithmetic of the fields cut out by mod $\ell$ Galois representations. 

This research was supported by the EPSRC Programme Grant EP/K034383/1 ``LMF: L-Functions and Modular Forms''. 

The computations presented in this paper were carried out on the Warwick mathematics institute computer cluster provided by the aforementioned EPSRC grant. The computer algebra packages used were \cite{Sage}, \cite{gp} and \cite{Magma}, and we were able to evaluate root discriminants of Galois closures of wildly ramified fields thanks to John Jones's and David Roberts's page \cite{GRD}.
\end{abstract}

\newpage

\section{Introduction}

Let $f = \sum_{n \geqslant 1} a_n q^n$ be a classical cusp form of level $N \in \N$, weight $k \geqslant 2$ and nebentypus $\eps$. We suppose that $f$ is an \emph{eigenform}, that is to say that $f$ is an eigenvector of the Hecke operator $T_n$ for all $n \in \N$, and that it is \emph{new}, which means that there is no eigenform of level less than $N$ which has the same eigenvalues as $f$. In this case, the coefficient $a_1$ of $f$ is necessarily nonzero, so we may normalise $f$ so that $a_1 = 1$. We call such a normalised new eigenform a \emph{newform}. For all $n \in \N$, the eigenvalue of $f$ with respect to the Hecke operator $T_n$ is the coefficient $a_n$ of $f$. These coefficients are algebraic integers and span a number field (that is to say an extension of $\Q$ of \emph{finite} degree) called the \emph{Hecke field} of $f$. This field also contains the values of $\eps$.

Furthermore, for all finite primes $\l$ of this field, there exists a unique \emph{mod $\l$ Galois representation} of degree $2$, that is to say a continuous group homomorphism
\[ \rho_{f,\l} : \Gal(\overline \Q / \Q) \longrightarrow \GL_2(\F_\l), \]
which sends any Frobenius element at $p$ to a matrix of characteristic polynomial
\[ T^2 - a_p T + p^{k-1} \eps(p) \in \F_\l[T] \]
for all primes $p \nmid \ell N$. Here, $\F_\l$ is the residue field of $\l$, $\ell \in \N$ is the prime below $\l$, and the coefficients of the characteristic polynomial are considered mod $\l$.

As $\rho_{f,\l}$ is continuous, Galois theory attaches to the kernel of $\rho_{f,\l}$ a number field $L$, which we call the field \emph{cut out} by $\rho_{f,\l}$. It is a Galois number field, whose Galois group is isomorphic via $\rho_{f,\l}$ to the image of $\rho_{f,\l}$, as shown in the diagram below.
\[ \xymatrix{
\Gal(\overline \Q / \Q) \ar[r]^{\rho_{f,\l}} \ar@{->>}[d] & \GL_2(\F_\l) \\
\Gal(L/\Q) \ar@{^(->}[ur]
} \]
 
The ramification properties of $L$ are well-understood in terms of $f$ and $\l$. In particular, $L$ is unramified at $p$ if $p \nmid \ell N$, and $L$ is at most tamely ramified at every $p \neq \ell$ such that $p \parallel N$.

At $\ell$, the field $L$ is usually wildly ramified, but not always. More precisely, it is tamely ramified when $f$ admits a \emph{companion form} mod $\ell$ in the sense of $\cite{Gross}$, or when $f$ is \emph{supersingular} at $\l$. The first case means that there exists another eigenform $g = \sum_{n \geqslant 1} b_n q^n$, of the same level as $f$ but of weight $\ell+1-k$, such that
\[ \sum_{n \geqslant 1} n a_n q^n \bmod \l = \sum_{n \geqslant 1} n^k b_n q^n \bmod \l' \]
for some prime $\l'$ above $\ell$ in the number field generated by the coefficients $b_n$. The second case means that the $\ell$-th Fourier coefficient $a_\ell$ of $f$ is $0 \bmod \l$; in this case, there also exists another eigenform $g = \sum_{n \geqslant 1} b_n q^n$, of the same level as $f$ but of weight $\ell+3-k$ this time, such that
\[ \sum_{n \geqslant 1} n^2 a_n q^n \bmod \l = \sum_{n \geqslant 1} n^k b_n q^n \bmod \l'. \]

Therefore, Galois representations attached to such forms are a valuable source of number fields of Galois group $\GLFl$ or $\PGLFl$ whose ramification is extremely restricted, and that thus deserve a particular place in tables of number fields, provided of course that we are able to find them explicitly. This was pointed out to the author by David Roberts. A quantitative statement of this fact is achieved by theorem \ref{thm:predict_disc}.

In previous works \cite{algo} and \cite{certif}, we described algorithms to compute explicitly the number field cut out by the mod $\l$ representation attached to a form of level $N=1$. In this article, we show how these algorithms can be generalised to forms of higher level $N$, provided for simplicity that $\ell N$ is squarefree. We lose no generality by assuming that $\ell$ and $N$ are coprime since every mod $\ell$ representation attached to a form of level $\ell N$ is also attached to a form of level $N$; besides, the hypothesis that $N$ is squarefree could probably be suppressed without great difficulty.

As an application, we compute the Galois representations attached to a few forms which admit a companion or are supersingular $\bmod~\ell$, and we obtain very lightly ramified number fields of degree $\ell+1$ that were, as far as the author knows, previously unknown, and whose Galois closure has Galois group $\PGLFl$ and a particularly small root discriminant, thus beating the record for such number fields.

As the output of this algorithm is not certified for the same reasons as in \cite{algo}, we explain how our certification method \cite{certif} can be extended to the case of forms of higher level.

Finally, we establish formulas to predict the discriminant of the number fields cut out by such representations and we use these formulas to single out several interesting examples, even though we are unable to compute these fields explicitly at the moment.

\bigskip

This article is organised as follows. First, in section \ref{sect:WQ}, we derive formulas describing the action of the Atkin-Lehner operators on a space of modular forms of given level and weight, including on the old subspace and on Eisenstein series. These formulas are needed for the generalisation of our modular Galois representation computation algorithm \cite{algo}, and we present this generalisation in section \ref{sect:algo}. In section \ref{sect:certif}, we explain how the output of this new algorithm may be certified thanks to a generalisation of the methods presented in \cite{certif}. Finally, in section \ref{sect:results} we present the results of our computations, and we explain in section \ref{sect:what_to_expect} what results can be obtained with our techniques in general.

\subsection*{Notation}

We let $G_\Q$ denote the absolute Galois group $\Gal(\overline \Q / \Q)$ of the rationals, and we write $\Frob_p \in G_\Q$ for a Frobenius element at the prime $p \in \N$. Since all the Galois representations considered in this article are mod $\ell$ (as opposed to $\ell$-adic), we will denote them by $\rho$ (as opposed to $\overline \rho$). We will also frequently consider projective mod~$\ell$ Galois representations, which we will denote by the letter $\pi$. This should hopefully not cause any confusion, as we will not consider any automorphic representation in this article.

We write $e(x)$ as a shorthand for $e^{2i \pi x}$.

Let $\eps$ be a Dirichlet character modulo $N \in \N$. We write
\[ \g(\eps) = \sum_{x \bmod N} \eps(x) e(x/N) \]
for the Gauss sum of $\eps$, which is thus $-1$ when $N$ is prime and $\eps$ is the trivial character mod $N$. Given a factorisation $N = Q_1 \cdots Q_r$ of $N$ into pairwise coprime factors $Q_i$, we will denote by $\eps_{Q_1} \cdots \eps_{Q_r}$ the corresponding decomposition of $\eps$ into characters of respective moduli $Q_i$. Also, if $N'$ is a multiple of $N$, we will denote by $\eps^{(N')}$ the character obtained from $\eps$ by raising the modulus to $N'$.

We normalise the weight $k$ action on functions on the upper half-plane as
\[ f \mact_k \smat{a}{b}{c}{d} (\tau) = \frac{(ad-bc)^{k/2}}{(c \tau+d)^k} f\left( \frac{a \tau + b}{c \tau + d}\right). \]
This is an action of $\PGL_2^+(\R)$.

For $k,N \in \N$ and $\eps$ a Dirichlet character mod $N$, we let $\M_k(N,\eps)$ \big(resp. $\S_k(N,\eps)$, $\Ei_k(N,\eps)$\big) be the $\C$-vector space of modular forms (resp. cuspforms, Eisenstein series) of level $N$ and nebentypus $\eps$, and we let $\newf_k(N,\eps)$ be the finite set of newforms in $\S_k(N,\eps)$.

When $N \mid N'$, we write $I_N^{N'}$ for the ``identity'' injection map from $\M_k(N,\eps)$ to $\M_k(N',\eps^{(N')})$.

Also, for $t \in \N$, we let $B_t$ denote the operator $f(\tau) \mapsto f(t \tau)$, in other words
\[ f \mact B_t = t^{-k/2} f \mact_k \smat{t}{}{}{1} \]
for $f$ of weight $k$. This operator sends $\S_k(N, \eps)$ to $\S_k(tN, \eps^{(tN)})$, and $\Ei_k(N, \eps)$ to $\Ei_k(tN, \eps^{(tN)})$. Note that $B_t B_{t'} = B_{t t'}$ for all $t,t' \in \N$.

Finally, we write $\langle d \rangle$ for the Hecke operator that acts as multiplication by $\eps(d)$ on $\M_k(N,\eps)$. In particular, $\langle d \rangle$ is the $0$ operator if $\gcd(d,N) > 1$.

\section{Explicit formulas for Atkin-Lehner operators acting on the whole space of modular forms}\label{sect:WQ}

In order to compute mod $\ell$ Galois representations attached to eigenforms of level $N > 1$, we will need to be able to compute the action of the Atkin-Lehner operators on the space $\M_k\big( \Gamma_1(\ell N) \big)$, including its old part and its Eisenstein part. The purpose of this section is to establish explicit formulas for this.

\subsection{Basic facts about the Atkin-Lehner operators}

To begin with, lets us recall without proof the definition and some well-known basic facts about the Atkin-Lehner operators. Proofs may be found in \cite{Asai} and \cite{AL78} for instance.

Fix an integer $N \in \N$, and let $Q \in \N$ be a divisor of $N$. We say that $Q$ is an \emph{exact divisor} of $N$, and we write $Q \parallel N$, if $\gcd(Q,N/Q)=1$.

Let $Q$ be an exact divisor of $N$, let $R = N/Q$, let $x$ and $y$ be two integers, and consider the set of matrices
\[ \W_{Q,x,y}^{(N)} = \Set{w = \smat{Qa}{b}{Nc}{Qd}}{a,b,c,d \in \Z,  a \equiv x \bmod R, b \equiv y \bmod Q, \det w = Q}. \]
We write $\W_{Q}^{(N)}$ for $\W_{Q,1,1}^{(N)}$, and we drop the superscript $(N)$ from the notation whenever $N$ is clear from context. Note that the condition $Q \parallel N$ ensures that $\W_{Q,x,y}^{(N)}$ is never empty as long as $x$ is coprime to $R$ and $y$ is coprime to $Q$.

Let $f \in \M_k(N,\eps)$. Then $f \mact_k w_Q$ does not depend on the choice of $w_Q \in \W_Q^{(N)}$, so we may write the corresponding operator as $W_Q$. More generally, for $w_Q \in \W_{Q,x,y}^{(N)}$ we have
\begin{equation} f \mact_k w_{Q} = \bar \eps_{R}(x) \bar \eps_Q(y) f \mact W_Q. \label{W_xy} \end{equation}
The operator $W_Q$ sends $\M_k(N,\eps)$ to $\M_k(N,\bar \eps_Q \eps_{R})$, $\S_k(N,\eps)$ to $\S_k(N,\bar \eps_Q \eps_{R})$, and $\Ei_k(N,\eps)$ to $\Ei_k(N,\bar \eps_Q \eps_{R})$.

The operators $W_Q$ for varying $Q \parallel N$ do not commute; however, if $Q$ and $Q'$ are two exact divisors of $N$ which are coprime, then we have the relation
\begin{equation} f \mact W_{Q Q'} = \eps_{Q'}(Q) f \mact W_Q \mact W_{Q'} \label{W_QQ'} \end{equation}
for all $f \in \M_k(N, \eps)$.

%

\subsection{Atkin-Lehner operators on the new part of the cuspidal subspace}

The explicit action of the operators $W_Q$ on newforms is well-understood. Indeed, we have the following formula (cf. \cite[section 2]{AL78} and \cite[theorem 2]{Asai}):

\begin{thm}
Let $f = q + \sum_{n \geqslant 2} a_n q^n \in \newf_k(N,\eps)$ be a newform, and let $Q \parallel N$. For all nonzero integers $n \in \N$, write $n_Q = \gcd(n,Q^\infty)$ for the part of $n$ that is supported by the primes dividing $Q$. Then there exists an algebraic number $\lambda_{f,Q} \in \C^*$ of absolute value $1$ such that
\[ f \mact W_Q = \lambda_{f,Q} \sum_{n \geqslant 1} b_n q^n, \]
where $b_n = \eps_{N/Q}(n_Q) \overline \eps_Q(n/n_Q) \overline{a_{n_Q} } a_{n/n_Q}$ for all $n \in \N$. In particular,
\[ f \mact W_N = \lambda_{f,N} \sum_{n \geqslant 1} \overline{a_n} q^n. \]
Let $Q = \prod_{i=1}^r q_i^{e_i}$ be the complete factorisation of $Q$, and write $Q_i = q_i^{e_i}$. Then the following conditions are equivalent:
\begin{itemize}
\item $a_Q \neq 0$,
\item $a_{Q_i} \neq 0$ for all $i$
\item $a_{q_i} \neq 0$ for all $i$,
\end{itemize}
and if these equivalent conditions are satisfied, then $\lambda_{f,Q}$ is given by
\[ \lambda_{f,Q} = \prod_{i=1}^r \eps_{Q_i}(Q/Q_i) \lambda_{f,Q_i}', \]
where
\[ \lambda_{f,Q_i}' = Q_i^{k/2-1}  \g(\eps_{Q_i}) / a_{Q_i} \]
and $\g(\eps_{Q_i})$ denotes the Gauss sum of $\eps_{Q_i}$.
\end{thm}

\begin{rk}
Note that according to \cite[theorem 3]{Li}, $a_{q_i}=0$ if and only if $e_i \geqslant 2$ and $\eps_{Q_i}$ is not a primitive character. Unfortunately, the author is not aware of any method to compute $\lambda_{f,Q}$ without summing infinite series numerically when this case occurs.
\end{rk}

\subsection{Atkin-Lehner operators on the new part of the Eisenstein subspace}

Fix an integer $k \geqslant 1$, and let $\psi$ and $\vphi$ be \emph{primitive} characters of respective moduli $u$ and $v$, and let $M = uv$. If $k=2$, also suppose that $M>1$. Let $z \in \C$ be such that $\operatorname{Re} z > 2-k$, and define
\[ G_{k,z}^{\psi,\vphi}(\tau) = \sum_{r \bmod u} \sum_{s \bmod v} \sum_{t \bmod u} \sum_{\substack{(c,d) \in \Z^2 \\ c \equiv rv \bmod M \\ d \equiv s+tv \bmod M}}  \frac{\psi(r) \bar \vphi(s)}{(c \tau+d)^k \vert c \tau + d\vert^z} \]
for $\tau$ in the upper half-plane.

This continues analytically to all $z \in \C$, so we may define
\[ G_{k}^{\psi,\vphi}(\tau) = \lim_{z \rightarrow 0} G_{k,z}^{\psi,\vphi}(\tau). \]
Then $G_{k}^{\psi,\vphi} \in \Ei_k(M,\psi \vphi)$ is an Eisenstein series, whose $q$-expansion is
\[ G_{k}^{\psi,\vphi} = \left(\frac{-2 \pi i}{v}\right)^k \frac{g(\bar \vphi)}{(k-1)!} E_{k}^{\psi,\vphi}, \]
\[ E_{k}^{\psi,\vphi} = - \charf_{u=1} \frac{B_{k,\vphi}}k + 2 \sum_{n=1}^{+\infty} \sum_{m \mid n} \psi(n/m) \vphi(m) m^{k-1} q^n, \]
where $B_{k,\vphi}$ is the $k$-th Bernoulli number attached to $\vphi$, which is defined by the identity
\[ \sum_{n=0}^{+\infty} B_{n,\vphi} \frac{T^n}{n!} = \sum_{s=0}^{v-1} \vphi(s) \frac{T e^{sT}}{e^{vT}-1}. \]
The series $G_k^{\psi,\vphi}$ is also an eigenform for the whole Hecke algebra.

\bigskip

If $k=2$, also define
\[ E_2(\tau) = 1 - 24 \sum_{n=1}^{+\infty} \sigma_1(n) q^n, \]
which is NOT a modular form, and
\[ E^{(M)}(\tau) = E_2(\tau) - M E_2(M\tau) \in \Ei_2\big(\Gamma_0(M)\big), \]
which is a modular form, and also an eigenform for the Hecke operators $T_p$ such that $p \nmid M$.

\pagebreak

It is well-known (cf. for instance \cite[theorems 4.5.2 and 4.6.2]{DS}) that for $k = 1$ or $k \geqslant 3$, the series $E_{k}^{\psi,\vphi} \mact B_t$ for $(\psi,\vphi,t)$ such that

\begin{itemize}
\item $\psi$ and $\vphi$ are primitive,
\item the product $t \cdot \cond(\psi) \cdot \cond(\vphi)$ divides $N$,
\item $\psi(-1)\vphi(-1) = (-1)^k$,
\item and $\psi^{(N)} \vphi^{(N)} = \eps$
\end{itemize}
form a basis of the Eisenstein space $\Ei_k(N,\eps)$. Similarly, for $k=2$ the series $E_{2}^{\psi,\vphi} \mact B_t$ for $(\psi,\vphi,t)$ such that
\begin{itemize}
\item $\psi$ and $\vphi$ are primitive and not both trivial,
\item the product $t \cdot \cond(\psi) \cdot \cond(\vphi)$ divides $N$,
\item $\psi(-1)\vphi(-1) = 1$,
\item and $\psi^{(N)} \vphi^{(N)} = \eps$,
\end{itemize}
and in addition the series $E^{(t)}$ for $1 < t \mid N$ if $\eps$ is trivial, form a basis of $\Ei_2(N,\eps)$.  In this way, the series $E_{k}^{\psi,\vphi}$ with $\cond(\psi) \cond(\vphi)=N$ and $\psi(-1)\vphi(-1) = (-1)^k$, as well as $E^{(N)}$ if $k=2$, play   in $\Ei_k(N,\eps)$ a r\^ole that is analogous to the newforms in $\S_k(N,\eps)$. As a consequence, we make the following \emph{ad hoc} definition:

\begin{de}
Let $\eps$ be a Dirichlet character mod $N$. The \emph{new part} of  the Eisenstein space $\Ei_k(N,\eps)$ is the linear span of the $E_{k}^{\psi,\vphi}$ for $(\psi,\vphi)$ ranging over the set of couples of primitive Dirichlet characters such that $\cond(\psi) \cond(\vphi)=N$ and that $\psi^{(N)} \vphi^{(N)} = \eps$, as well as of $E^{(N)}$ if $k=2$ and $\eps$ is the trivial character mod $N$.
\end{de}

It follows from the above that this new subspace is also a Hecke submodule, unless $k=2$ and $\eps$ is trivial, in which case it is still a submodule over the algebra spanned by the $T_p$ for $p \nmid N$.

Continuing the parallel with the case of the cuspidal subspace, we are now going to derive formulas describing the action of the Atkin-Lehner operators $W_Q$ on this new subspace. 

\begin{thm}
Let $\psi$ and $\vphi$ be \emph{primitive} characters of respective moduli $u$ and $v$, let $N = uv$, and let $Q \parallel N$. Write $R = N/Q$, $u = u_Q u_R$ and $v = v_Q v_R$, where $u_Q = \gcd(u,Q)$, $u_R=\gcd(u,R)$, and similarly for $v_Q$ and $v_R$. Finally, let $\chi = \psi \vphi$, a character modulo $N$. Then
\[ G_k^{\psi,\vphi} \mact W_Q = \frac{(u_Q/v_Q)^{k/2} \vphi_Q(-1)}{\chi_Q(v_R)  \chi_R(v_Q)} G_k^{\bar \vphi_Q \psi_R, \bar \psi_Q \vphi_R}. \]
\end{thm}

\begin{proof}
Suppose first that $k \geqslant 3$, so that we may set $z=0$ in the formula defining $G_{k,z}^{\psi,\vphi}$ without having to invoke analytic continuation. The following formula is easily derived:

\[ G_{k}^{\psi,\vphi}(\tau) = \sum_{(m,n) \in \Z^2}' \frac{\psi(m) \bar\vphi(n)}{(m v \tau + n)^k}. \]

Here and in what follows, the dash sign on the double sum means that the term of index $m=n=0$ is omitted.

\bigskip

Let now $\smat{Qa}{b}{Nc}{Qd} \in \W_Q$, and compute that
\begin{align*}
(G_k^{\psi,\vphi} \mact W_Q)(\tau) =& Q^{k/2} \sum_{(m,n) \in \Z^2}' \frac{\psi(m) \bar\vphi(n)}{\big(m v (Qa \tau +b) + n (Nc \tau + Qd)\big)^k} \\
=& Q^{k/2} \sum_{(m,n) \in \Z^2}' \frac{\psi(m) \bar\vphi(n)}{\big((mvQa+nNc) \tau + (mvb+nQd)\big)^k} \\
=& Q^{k/2} \sum_{(m,n) \in \Z^2}' \frac{\psi(m) \bar\vphi(n)}{\big((m v_Q a+n u_R c) v_R Q \tau + (m v_R b+n u_Q d) v_Q \big)^k}
\end{align*}
\\
since $Q = u_Q v_Q$ and $R = u_R v_R$. Observe now that $\smat{v_q a}{v_R b}{u_R c}{u_Q d} \in \SL_2(\Z)$. As the double series is absolutely convergent, we may thus reindex it to get
\begin{align*}
(G_k^{\psi,\vphi} \mact W_Q)(\tau) =& Q^{k/2} \sum_{(m,n) \in \Z^2}' \frac{\psi(m u_Q d - n u_R c) \bar\vphi(-m v_R b + n v_Q a)}{\big(m v_R Q \tau + n v_Q \big)^k} \\
=& \frac{Q^{k/2}}{v_Q^k} \sum_{(m,n) \in \Z^2}' \frac{\psi(m u_Q d - n u_R c) \bar\vphi(-m v_R b + n v_Q a)}{\big(m v_R u_Q \tau + n\big)^k} \\
=& \left(\frac{u_Q}{v_Q}\right)^{k/2} \sum_{(m,n) \in \Z^2}' \frac{\psi_Q( - n u_R c) \psi_R(m u_Q d) \bar\vphi_Q(-m v_R b) \bar\vphi_R(n v_Q a)}{\big(m v_R u_Q \tau + n\big)^k} \\
=& \left(\frac{u_Q}{v_Q}\right)^{k/2} \psi_Q( -u_R c) \psi_R(u_Q d) \bar\vphi_Q(-v_R b) \bar\vphi_R(v_Q a) \sum_{(m,n) \in \Z^2}' \frac{\bar\vphi_Q \psi_R(m) \psi_Q \bar\vphi_R(n)}{\big(m v_R u_Q \tau + n\big)^k} \\
=& \left(\frac{u_Q}{v_Q}\right)^{k/2} \bar\psi_Q(v_R) \bar\psi_R(v_Q) \bar\vphi_Q(-v_R) \bar\vphi_R(v_Q) G_k^{\bar \vphi_Q \psi_R, \bar \psi_Q \vphi_R}(\tau),
\end{align*}
as the relations $a \equiv 1 \bmod R$, $b \equiv 1 \bmod Q$ and $Qad-Rbc=1$ from the definition of $\W_Q$ imply that $u_Q v_Q d \equiv 1 \bmod R$ and that $-u_R v_R c \equiv 1 \bmod Q$.

This concludes the proof when $k \geqslant 3$. For $k=1$ or $2$, by replacing the denominators appropriately (e.g. $(c\tau+d)^k$ becomes $(c\tau+d)^k \vert c\tau+d\vert^z$) throughout the computations, one obtains an identity between $G_{k,z}^{\psi,\vphi} \mact W_Q$ and $G_{k,z}^{\bar \vphi_Q \psi_R, \bar \psi_Q \vphi_R}$ which is valid for $\operatorname{Re} z + k > 2$. As both sides are analytic in $z$, this identity extends to all $z \in \C$, so we may let $z$ tend to $0$ to get the result.
\end{proof}

The case of the series $E^{(N)}$ is much simpler:

\begin{thm}
Let $Q \parallel N$, and let $R=N/Q$. Then $E^{(N)} \mact W_Q = E^{(R)} - E^{(Q)}$. \newline In particular, $E^{(N)} \mact W_N = - E^{(N)}$.
\end{thm}

\begin{proof}
Let us begin with the case $Q=N$. According to \cite[formula (1.5) page 19]{DS}, we have
\[ \frac1{\tau^2} E_2(-1/\tau) = E_2(\tau) + \frac{6}{\pi i \tau}, \]
so
\begin{align*}
\left( E^{(N)} \mact W_N \right) (\tau) &= \frac{N}{(N\tau)^2} \big( E_2(-1/N\tau) - N E_2(-N/N\tau) \big) \\
&= \frac{N}{(N\tau)^2} E_2(-1/N\tau) - \frac1{\tau^2} E_2(-1/\tau) \\
&= N \left(E_2(N\tau) + \frac{6}{\pi i N\tau} \right) - \big(E_2(\tau) + \frac{6}{\pi i \tau}\big) \\
&= -E^{(N)}(\tau).
\end{align*}

Then, in the case of general $Q \parallel N$, notice that
\[ E^{(N)}(\tau) = E_2(\tau) - Q E_2(Q \tau) + Q E_2(Q \tau) - N E_2(N \tau) = E^{(Q)}(\tau) + Q E^{(R)}(Q \tau), \]
so that
\[ E^{(N)} \mact W_Q = E^{(Q)} \mact W_Q + Q E^{(R)} \mact B_Q \mact W_Q = - E^{(Q)} + \frac{Q}{Q} E^{(R)} \]
by theorem \ref{thm:Wold} below and the inclusion $\W_Q^{(N)} \subset \W_Q^{(Q)}$ of matrices sets.
\end{proof}

\subsection{Atkin-Lehner operators on the old subspace}

We now derive formulas for the action of $W_Q$ on the old subspace of $\M_k(N,\eps)$. This involves computing the action of $W_Q$ (as an operator of level $N$) on forms which are new of level $M \mid N$ on the one hand, and deriving commutation relations between $W_Q$ and $B_t$ on the other hand.

In order to make the notation precise, we will specify the level at which $W_Q$ acts whenever necessary, by writing $W_Q^{(M)}$ for the operator $W_Q$ acting on a space of forms of level $M$. Recall that $I_M^N$ denotes the identity injection map from $\M_k(M,\eps)$ to $\M_k(N',\eps^{(N)})$ whenever $M \mid N$. Thus for instance one of our tasks is to compare $f \mact I_M^N \mact W_Q^{(N)}$ with $f \mact W_Q^{(M)} \mact I_M^N$ for $f$ of level $M$. This part could hence be seen as the study of the commutation relations between the $W_Q$ operators and other operators such as $B_t$ or $I_M^N$.

\bigskip

\begin{thm}\label{thm:Wold}
Let $M \in \N$, let $\eps$ be a Dirichlet character modulo $M$, and let \linebreak $f \in \M_k(M,\eps)$. Let $N$ be a multiple of $M$, and let $t$ divide $N/M$, so that $N=tMR$ for some integer $R \in \N$ and that $f \mact B_t$ may be seen as a form of level $N$.

Let $Q \parallel N$, and define $M_Q=\gcd(M,Q)$, $t_Q = \gcd(t,Q)$, $R_Q=\gcd(R,Q)$, so that $Q = t_Q M_Q R_Q$, and then let $M'=M/M_Q$ and $t'=t/t_Q$. Then the form $f \mact B_t \mact I_{tM}^N \mact W_Q^{(N)}$ depends on $Q$ but not on $N$ (as long as $N$ is such that $tM \mid N$ and that $Q \parallel N$ of course), so we may write it as $f \mact B_t \mact W_Q$. Explicitly, we have
\[ f \mact B_t \mact W_Q = (R_Q/t_q)^{k/2} \overline \eps_{M'}(t_Q) \overline \eps_{M_Q}(t') f \mact W_{M_q}^{(M)} \mact B_{t' R_Q}. \] 
\end{thm}

\begin{proof}
Let $w_Q = \smat{Qa}{b}{Nc}{Qd} \in \W_{Q}^{(N)}$, so that $a \equiv 1 \bmod N/Q$ and $b \equiv 1 \bmod Q$. Then we have
\begin{align*}
f \mact B_t \mact I_{tM}^N \mact W_Q^{(N)} &= t^{-k/2} f \mact_k \smat{t}{}{}{1} w_Q \\
&= t^{-k/2} f \mact _k\smat{t_Q}{}{}{1} \smat{t'}{}{}{1} w_Q \smat{t'^{-1}}{}{}{1} \smat{t'}{}{}{1} \\
&= t_Q^{-k/2} f \mact_k \smat{t_Q Qa}{t_Q t' b}{\frac{N}{t'}c}{Qd} \mact B_{t'} \\
&= t_Q^{-k/2} f \mact_k \smat{Q a}{t' b}{\frac{N}{t}c}{\frac{Q}{t_Q}d} \mact B_{t'} \\
&= t_Q^{-k/2} f \mact_k \smat{\frac{Q}{R_Q} a}{t' b}{\frac{N}{tR_Q}c}{\frac{Q}{t_Q}d} \smat{R_Q}{}{}{1} \mact B_{t'} \\
&= (R_Q/t_Q)^{k/2} f \mact_k \smat{M_Q t_Q a}{t' b}{M R' c}{M_Q R_Q d} \mact B_{t' R_Q}, \\
\end{align*}
where $R' = R/R_Q \in \N$. As $\smat{M_Q t_Q a}{t' b}{M R' c}{M_Q R_Q d} \in \W_{M_q,t_Q,t'}^{(M)}$, the result follows.
\end{proof}

This formula, along with the ones for newforms and Eisenstein series presented above, allow us to compute the action of $W_Q$ on the whole space $\M_k(N,\eps)$.

\section{Computation of modular Galois representations}\label{sect:algo}


In this section, we fix a prime $\ell \in \N$, a newform
\[ f = q +\sum_{n \geqslant 2} a_n q^n \in \S_k(N,\eps) \]
of weight $2 \leqslant k \leqslant \ell$, where $\eps$ is a Dirichlet character mod $N$, and a prime $\l$ \linebreak above $\ell \in \N$ of the number field generated by the coefficients $a_n$ of $f$. We want to compute the Galois representation
\[ \rho_{f,\l} : G_\Q \longrightarrow \GL_2(\F_\l) \]
attached to $f \bmod \l$.

In this article, we are especially interested in the case where $f \bmod \l$ admits a companion form or is supersingular, but the algorithms that we describe do not require that this is the case.

For simplicity, we will assume that the integer $\ell N$ is squarefree. This is mainly so as to simplify statements pertaining to the $q$-expansion of modular forms at all the cusps or such as proposition \ref{prop:periods} below, and this hypothesis could be suppressed without much difficulty.

Later on, we will focus on the case when the nebentypus $\eps$ of $f$ is trivial. This is only for the exposition's sake, and this hypothesis may also be removed very easily.

\subsection{The modular curve $X_H(\ell N)$}\label{sect:XH}

Just as in \cite{algo}, the idea of our algorithm is to ``catch'' the representation $\rho_{f,\l}$ in the torsion of the jacobian of a modular curve. More precisely, according to \cite[theorem 9.3 part 2]{Gross}, there exists an eigenform $f_2$ of weight $2$ and level $\Gamma_1(\ell N)$, a prime $\l_2 \mid \ell$ of the Hecke field of $f_2$ and an identification of residue fields $\F_\l \simeq \F_{\l_2}$ such that
\[ f \bmod \l = f_2 \bmod \l_2. \]
Let $\eps_2$ be the nebentypus of $f_2$, which is a Dirichlet character modulo $\ell N$. The same reference also tells us that $f_2$ may be chosen so that the $N$-part of $\eps_2$ agrees with $\eps$, in equations
\begin{equation} (\eps_2)_N = \eps. \label{char2_compat} \end{equation}
We will use this relation later; for now, the principal consequence of all this is that $\rho_{f,\l} \sim \rho_{f_2,\l_2}$ is afforded in the $\ell$-torsion of the jacobian of the modular curve $X_1(\ell N)$.

However, the genus of $X_1(M)$ grows quickly\footnote{It is roughly $M^2/24$ by Riemann-Hurwitz.} with $M$, and unfortunately the algorithm \cite{algo} cannot reasonably cope with genera higher than 30. As a result, we are limited to $\ell \leqslant 31$ when $N=1$, and to even smaller values of $\ell$ when $N$ is larger.

Nevertheless, the representation we are interested in is afforded in the abelian variety $A_{f_2}$ corresponding to the Galois orbit of the eigenform $f_2$, which is a factor (up to isogeny) of the jacobian $J_1(\ell N)$ of $X_1(\ell N)$, and it is quite possible that $A_{f_2}$ is much smaller than $J_1(\ell N)$. Unfortunately, we do not know how to compute explicitly with abelian varieties, unless they are provided to us as the jacobian of some curve. We thus want to find a curve whose jacobian contains $A_{f_2}$ and is not much larger than it.

A natural solution, which we owe to \cite{Maarten}, is to introduce the modular curve $X_H(\ell N)$ corresponding to the congruence subgroup
\[ \Gamma_H(\ell N) = \left\{ \smat{a}{b}{c}{d} \in \SL_2(\Z) \ \big\vert \ c \mid \ell N \text{ and } a,d \bmod \ell N \in H \right\}, \]
where $H \leqslant (\Z/ \ell N \Z)^*$ is the kernel of the nebentypus $\eps_2$ of $f_2$. As $\Gamma_H(\ell N)$ is an intermediate congruence subgroup between $\Gamma_1(\ell N)$ and $\Gamma_0(\ell N)$, this modular curve is defined over $\Q$ and is intermediate between $X_0(\ell N)$ and $X_1(\ell N)$. In some cases, its genus is significantly smaller than that of $X_1(\ell N)$, and so we save a lot of computational effort by replacing $X_1(\ell N)$ with it, but in other cases we have $H \leqslant \{ \pm 1 \}$ so that $\Gamma_H(\ell N) = \Gamma_1(\ell N)$ and $X_H(\ell N) = X_1(\ell N)$.

More precisely, note that the respective determinant characters of $\rho_{f,\l}$ and of $\rho_{f_2,\l_2}$ are
\[ \det \rho_{f,\l} : \begin{array}{rcl}
 G_\Q & \longrightarrow & \F_\l^* \\
 \Frob_p & \longmapsto & p^{k-1} \eps(p) \bmod \l \\
 \end{array} \]
 and 
\[ \det \rho_{f_2,\l_2} : \begin{array}{rcl}
 G_\Q & \longrightarrow & \F_{\l_2}^* \\
 \Frob_p & \longmapsto & p^{2-1} \eps_2(p) \bmod \l_2 \\
 \end{array} \]
 for $p \nmid \ell N$ prime, so since these representations agree, the character $\eps_2$ must satisfy 
 \[ \eps_2(x) \bmod \l_2 = x^{k-2} \eps(x) \bmod \l \]
 for all $x \in \Z$ by Dirichlet's theorem on arithmetic progressions.
 
To simplify, \textbf{we suppose from now on that the nebentypus $\eps$ of $f$ is trivial}. By \eqref{char2_compat}, $\eps_2$ is then a Dirichlet character modulo $\ell N$ of conductor $\ell$ or $1$, and it satisfies
\[ \eps_2(x) \equiv x^{k-2} \bmod \l_2 \]
for all $x \in \Z$, so that the subgroup $H \leqslant (\Z/\ell N \Z)^*$ is the pull-back to $(\Z/\ell N \Z)^*$ of the subgroup $K$ of $(\Z/\ell \Z)^*$ formed of the $(k-2)$-torsion elements. As a result, our curve $X_H(\ell N)$ may be seen as the fibred product
\begin{equation} X_H(\ell N) = X_K(\ell) \times_{X(1)} X_0(N); \label{X_H_fibred} \end{equation}
besides, the lower the multiplicative order of $k-2 \bmod \ell-1$, the larger $H$, hence the smaller $X_H(\ell N)$ and the more efficient our computation of $\rho_{f,\l}$ will be. For instance, for $k=2$, we have $X_H(\ell N) = X_0(\ell N)$, which reminds us that $\rho_{f,\l}$ is already afforded in the torsion of the jacobian of $X_0(N)$ and that we have raised the level to $\ell N$ for nothing in this case. The next best case is $k-2 = (\ell-1)/2$, for which $X_H(\ell N)$ is a double cover of $X_0(\ell N)$ and $\eps_2$ is the Legendre symbol at $\ell$ seen as a character mod $\ell N$.
 
\subsection{The periods of $X_H(\ell N)$}

As in \cite{algo}, in order to compute $\rho_{f,\l}$ we begin by computing the periods of the modular curve $X_H(\ell N)$, which will allow us to view its jacobian as an explicit complex torus.

For each  Dirichlet character $\chi$, define the modular symbol
\[ s_\chi = \sum_{a \bmod m} \bar \chi(-a) \{ \infty , a/m \}, \]
where $m$ is the modulus of $\chi$.  This well-defined, as
\[ \int_{x}^{x+1} F = \int_{x}^{\infty} F - \int_{x+1}^\infty F = 0 \]
for all $x \in \R$ and every cuspform $F$ by 1-periodicity of $F$.

Let $(\gamma_j)_{1 \leqslant j \leqslant 2g}$ be a $\Z$-basis of the homology of $X_H(\ell N)$, and let $T$ be a Hecke operator at level $\Gamma_H(\ell N)$.  We express each $\gamma_j$ as a linear combination
\[ \gamma_j = \sum_{i,\chi} \lambda_{i,\chi} T^i s_\chi \]
of modular symbols $T^i s_\chi$ over some appropriate cyclotomic field $K$, where the characters $\chi$ are allowed to have different moduli, but are \emph{all primitive}. We do so by starting with $m=1$, and increasing $m$ until the $T^i s_\chi$ for $\chi$ primitive of modulus at most $m$ span the homology of $X_H(\ell N)$ over $K$. In general, characters $\chi$ of large modulus mean that more $q$-expansion terms are required to compute the integral of a cuspform along $s_\chi$ (cf. proposition \ref{prop:periods} below), so if possible we choose $T$ to be a generator of the Hecke algebra as a $\Q$-algebra, so as to minimise the number of values of $m$ we have to try until we get a generating set.

\bigskip

The space of holomorphic differentials on $X_H(\ell N)$ is
\begin{equation} \S_2\big(\Gamma_H(\ell N)\big) = \bigoplus_{\substack{\eps \bmod \ell N \\ \Ker \eps \geqslant H}} \S_2(\ell N, \eps). \label{dec_S2} \end{equation}
A natural basis of this space is formed of the $f \mact B_t$, where $f \in \newf_2(M,\eps)$ is a newform whose level $M$ divides $\ell N$ and whose nebentypus $\eps$ factors through $H$, and $t$ divides $\ell N/M$. The entries of the corresponding period matrix are of the form
\[ \int_{\gamma_j} f \mact B_t =  \sum_{i,\chi} \lambda_{i,\chi} \int_{s_\chi} f \mact B_t \mact T^i, \]
so to compute these entries we need to know how to compute explicitly the action of $T$ on the stable subspace spanned by the $f \mact B_t$ for fixed $f$ and varying $t$ on the one hand, and the integrals $\int_{s_\chi} f \mact B_t$ on the other hand. This is what we achieve in the following two propositions. 

\begin{pro}
Let $n \in \N$ be any integer, and let $T_n$ be the corresponding Hecke operator at level $\Gamma_H(\ell N)$. Let $f = q + \sum_{m \geqslant 2} a_m q^m \in \newf_k(M,\eps)$ be a newform of weight $k$ whose level $M$ divides $\ell N$ and whose nebentypus $\eps$ factors through $H$, and let $t$ be a divisor of $\ell N / M$. Write $n = n_1 n_2$, where $n_1 = \gcd(n,(\ell N)^\infty)$, and factor $n_1$ as $\prod_i p_i^{e_i}$. Then
\[ f \mact B_t \mact I_{tM}^{\ell N} \mact T_n = a_{n_2} f \mact B_t \mact \prod_i U_{p_i}^{e_i}, \]
where $U_p$ is the operator
\[ \sum_m b_m q^m \longmapsto \sum_m b_{pm} q^m, \]
and furthermore
\[ f \mact B_t \mact U_p = \left\{ \begin{array}{ll} a_p f \mact B_{t} - p^{k-1} \eps(p) f \mact B_{pt} & \text{if } p \nmid t, \\ f \mact B_{t/p} & \text{if } p \mid t  \end{array} \right. \]
for all primes $p \in \N$. Note that in the first case, $\eps(p)=0$ if $p \mid M$, since $\eps$ is a character mod $M$.
\end{pro}

\begin{proof}
Immediate from the formulas
\[ F \mact T_{n_1} T_{n_2} = F \mact T_{n_2} T_{n_1} \quad \text{if } \gcd(n_1,n_2)=1, \]
\[ F \mact U_{p^e} = F \mact U_p^e \quad \text{for } p \mid \ell N \text{ prime and } e \in \N, \]
and
\[ F \mact T_p = \sum_m b_{pm} q^m + p^{k-1} \chi(p) \sum_m b_m q^{pm} \]
valid for all $F = \sum_m b_m q^m \in \S_k(\ell N,\chi)$.
\end{proof}

\begin{pro}\label{prop:periods}
Fix a squarefree integer $N \in \N$, consider a newform
\[ f  = q + \sum_{n \geqslant 2} a_n q^n \in \newf_2(M,\eps) \]
of level $M \mid N$ and nebentypus $\eps$, let $\lambda \in \C^*$ be such that
\[ f \mact W_M = \lambda \sum_{n \geqslant 1} \bar a_n q^n, \]
and let $t \mid \frac{N}{M}$, so that $f \mact B_t$ is a cuspform of level $N$. Then for all \emph{primitive} Dirichlet characters $\chi$ whose modulus $m$ is prime to $N$, we have
\[ \int_{s_\chi} f \mact B_t = \frac{\chi(t)}t \frac{m}{2 \pi i \g(\chi)} \sum_{n \geqslant 1} \frac{\chi(n) a_n - \lambda_\chi \bar \chi(n) \bar a_n}n R^n, \]
where $\g(\chi)$ denotes the Gauss sum of $\chi$, $\lambda_\chi = \chi(-M) \eps(m) \frac{\g(\chi)}{\g(\bar \chi)} \lambda$, and $R = e^{-2 \pi / m \sqrt{M}}$.
\end{pro}

\newpage

\begin{proof}
By Fourier analysis, we have
\[ \chi(x) = \frac1m \sum_{a \bmod m} \hat \chi(a) e\left(\frac{ax}m \right), \]
where
\[ \hat \chi(a) =\sum_{x \bmod m} \chi(x) e\left(-\frac{ax}m \right) = \bar \chi(-a) \g(\chi) \]
since $\chi$ is primitive. Therefore,
\begin{align*}
\int_{s_\chi} f \mact B_t =& \sum_{x \bmod m} \bar \chi(-x) \int_{\infty}^{x/m} f(t \tau) d \tau \\
=& \frac1t \sum_{x \bmod m} \bar \chi(-x) \int_{\infty}^{tx/m} f(\tau) d \tau = \frac{\chi(t)}t \sum_{x \bmod m} \bar \chi(-x) \int_{\infty}^{x/m} f(\tau) d \tau
\end{align*}
by the changes of variable $\tau' = t \tau$ and $x' = tx$, which is legitimate since $t$ and $m$ are coprime.

Next, we have
\begin{align*}
\sum_{x \bmod m} \bar \chi(-x) \int_{\infty}^{x/m} f(\tau) d \tau &= \sum_{x \bmod m} \int_{\infty}^{0} \bar \chi(-x) f\left(\tau+\frac{x}{m}\right) d \tau \\
&= \int_{\infty}^0 \sum_{n \geqslant 1} \sum_{x \bmod m} \bar \chi(-x) e(n x/m) a_n e(n \tau) d\tau \\
&= \frac{m}{\g(\chi)} \int_{\infty}^0 \sum_{n \geqslant 1} \chi(n) a_n e(n \tau) d\tau \\
&= \frac{m}{\g(\chi)} \int_{\infty}^0 f \otimes \chi,
\end{align*}
where $f \otimes \chi = q + \sum_{n \geqslant 2} \chi(n) a_n q^n$ is a newform of weight 2, level $M' = m^2 M$ and character $\eps \chi^2$ by \cite[p. 228]{AL78}. Furthermore, according to the same reference, we have
\[ f \otimes \chi \mact W_{M'} = \lambda_\chi \left( q + \sum_{n \geqslant 2} \bar \chi(n)  \bar{a_n} q^n \right). \]
As a result, we have
\begin{align*}
\int_{\infty}^0 f \otimes \chi &= \left( \int_{\infty}^{\frac{i}{\sqrt{M'}}} + \int_{\frac{i}{\sqrt{M'}}}^{0} \right) f \otimes \chi \\
&= \int_{\infty}^{\frac{i}{\sqrt{M'}}} \left( f \otimes \chi - f \otimes \chi \mact W_{M'} \right) \\
&= \frac1{2 \pi i} \sum_{n\geqslant 1} (\chi(n) a_n - \lambda_\chi \bar \chi(n) \bar a_n) \frac{e^{-2\pi n/\sqrt{M'}}}n. 
\end{align*}
\end{proof}

\subsection{High-precision $q$-expansion of the forms of weight 2}

Thanks to the previous two formulas, we are thus able to compute the periods of $X_H(\ell N)$ to very high precision, provided that we first compute enough terms of the $q$-expansion of the cuspforms of weight 2 forming our basis \eqref{dec_S2} of $\S_2\big( \Gamma_H(\ell N) \big)$, which is a non-trivial task since we typically need a few hundred thousands terms. We do so thanks to an improved version of the mod $p$ modular equation method which we described in \cite[section 3.1]{algo}.

Suppose for the simplicity of the exposition that the dimension $g_0$ of $\S_2\big( \Gamma_0(\ell N) \big)$ is at least $2$, and let $f_1, \dots , f_{g_0} \in \S_2\big( \Gamma_0(\ell N) \big)$ be a basis made of forms whose Fourier coefficients are rational, where $g_0$ is thus the genus of $X_0(\ell N)$. Thanks to the Sturm bound, we may easily rescale these forms so that their Fourier coefficients are integers. As in \cite[section 3.1]{algo}, we then begin by computing a polynomial equation relating the functions $f_1 / du$ and $u$ on $X_0(\ell N)$ modulo a large enough prime $p \in \N$, where $u = 1 / j \in q \Z [[ q]]$ is the multiplicative inverse of the $j$-invariant. The degrees of this equation are respectively $d_0$ and at most $d_0+s_0+2g_0-2$, where $d_0$ is the index of $\Gamma_0(\ell N)$ in $\SL_2(\Z)$ and $s_0$ is the number of cusps of $X_0(\ell N)$. We then compute a lot of terms of the $q$-expansion of $u$ mod $p$. This can be done quickly thanks to the formulas
\[ u = \frac{E_4^3-E_6^2}{(12 E_4)^3}, \]
\[ E_4 = 1 + 240 \sum_{n \geqslant 1} \left( \sum_{0 < d \mid n} d^3 \right) q^n, \]
\[ E_6 = 1 - 504 \sum_{n \geqslant 1} \left( \sum_{0 < d \mid n} d^5 \right) q^n, \]
and an Eratosthenes sieve, as long as we use fast series arithmetic and we do all the computations mod $p$. Next, we use Newton iteration in $\F_p[[q]]$ on the bivariate polynomial equation so as to recover the coefficients mod $p$ of $f_1/du$, and hence of $f_1$. Finally, we lift these coefficients back to $\Z$, which we can do unambiguously thanks to Deligne's bounds if $p$ is large enough.

This is rather slow, as the degrees of the polynomial equation tend to be high, so for $f_2$ we proceed a bit differently, by computing a polynomial equation relating $f_2 / f_1$ and $u$ modulo a (possibly different) large prime $p$. We then use Newton iteration to recover the coefficients of $f_2$. This time, the degrees of this equation are $d_0$ and at most $2g_0-2$, which is already much better, so this is much faster than for $f_1$.

Then, for all the other forms $f_i$ in the basis of  $\S_2\big( \Gamma_0(\ell N) \big)$, we compute an equation relating $f_i / f_1$ to $f_2 / f_1$ modulo a (possibly again different) large prime $p$, and we use Newton iteration to deduce the coefficients of $f_i$. This is very fast, as the degree of this equation is at most $2g_0-2$ in each variable; besides, all the forms $f_i$ may be treated in parallel.

Finally, let $\eps$ be a non-trivial character appearing in \eqref{dec_S2}, let $r > 1$ be its order, and fix a basis $F_1, \dots, F_d$ of $\S_2(\ell N, \eps)$ made up of forms whose Fourier coefficients lie in the value field $K$ of $\eps$. Such a basis always exists and may easily be computed thanks to \cite[theorem p. xiii]{CasselsFlynn}. Again, thanks to the Sturm bound, we may effortlessly arrange for the $F_i$ to have integral coefficients. Then, for each $F_i$, we choose a large enough prime $p$ such that $p \equiv 1 \bmod r$, so that $p$ splits completely in $K$, and we compute an equation relating $(F_i/f_1)^r \bmod \p$ to $f_2 / f_1 \bmod p$ for each prime $\p$ of $K$ above $p$. By definition of $r$, the function $(F_i/f_1)^r$ descends to $X_0(\ell N)$, so this equation has degrees at most $2g_0-2$ and $2g-2$, where $g$ is the genus of $X_H(\ell N)$, and so the computation is still reasonably fast. We then use Newton iteration to recover the coefficients of $F_i \bmod \p$, and finally lift these coefficients back to $K$ thanks to Chinese remainders over the primes $\p \mid p$. Note that the various forms $F_i$ and primes $\p$ may easily be processed in parallel.

As in \cite[section 3.1]{algo}, we thus obtain a method to expand a basis of $\S_2\big( \Gamma_0(\ell N) \big)$ to $q$-adic accuracy $O(q^B)$ in time quasilinear in $B$. However, this new method performs much better in practice, since it relies on modular equations of degrees much smaller than in \cite{algo} for all but the first form.

\subsection{The rest of the computation}

Once we have computed a very precise approximation of the periods of $X_H(\ell N)$ over $\C$, we may proceed essentially as in \cite{algo}, by inverting the Abel-Jacobi map at $\ell$-torsion points thanks to Kamal Khuri-Makdisi's algorithms \cite{Kamal}.

In order to adapt these algorithms to $X_H(\ell N)$, we need to compute the Riemann-Roch space
\begin{equation} V_2 \simeq H^0\big(X_H(\ell N), D_0 \big) \label{V2_H0} \end{equation}
attached to a divisor $D_0$ defined over $\Q$ whose degree $d_0$ is at least $2g+1$, where $g$ is the genus of $X_H(\ell N)$. As in \cite{algo}, we let $D_0$ be the sum of a canonical divisor of $X_H(\ell N)$ and of a divisor $D_\infty$ of degree $3$ supported by three distinct cusps, so that $d_0 = 2g+1$ exactly, as higher values of $d_0$ would just slow Makdisi's algorithms down. Thus $V_2$ is the space of meromorphic differentials that have at most simple poles at the cusps supporting $D_\infty$ and are holomorphic elsewhere. As a result, $V_2$ is contained in the space $\M_2\big(\Gamma_H(\ell N)\big)$ of modular forms of weight $2$; more precisely, we have
\[ V_2 = \S_2\big(\Gamma_H(\ell N)\big) \oplus E, \]
where $E$ is the subspace of dimension\footnote{As shown by \eqref{V2_H0}, the fact that the dimension of $E$ is $2$ is a consequence of the Riemann-Roch theorem. Alternatively, this can be seen directly since the map that evaluates the modular forms at all of the cusps induces an isomorphism between the Eisenstein subspace and the trace zero subspace.} $2$ of the Eisenstein space $\Ei_2\big(\Gamma_H(\ell N)\big)$ formed of the series that vanish at all the cusps except those supporting $D_\infty$.

In order for $D_0$ to be defined over $\Q$, we need $D_\infty$ to be defined over $\Q$ itself. Since we assumed that $N$ is squarefree, say $N= p_1 \cdots p_r$, and that $\eps$ is trivial, this is not difficult, as by \eqref{X_H_fibred} we have
\[ X_H(\ell N) = X_K(\ell) \times_{X(1)} X_0(p_1)  \times_{X(1)} \cdots  \times_{X(1)} X_0(p_r) \]
where $K$ is the subgroup of $(\Z/\ell \Z)^*$ formed of the elements that are killed by $k-2$. Indeed, since $X(1)$ has a single cusp, this decomposition therefore yields an identification of $G_\Q$-sets
\[ \text{Cusps}\big(X_H(\ell N)\big) =  \text{Cusps}\big(X_K(\ell)\big) \times \text{Cusps}\big(X_0(p_1)\big)  \times \cdots  \times \text{Cusps}\big(X_0(p_r)\big). \]
For all $p \in \N$ prime, $X_0(p)$ has exactly two cusps, namely $\infty$ and $0$, and both are defined over $\Q$, whereas for all odd $p \in \N$, $X_1(p)$ has exactly $p-1$ cusps, that form two orbits of size $(p-1)/2$ under the $\langle d \rangle$ operators. One of these orbits is formed of cusps that project to the cusp $0$ of $X_0(p)$, and all these cusps are defined over $\Q$, whereas the other is formed of cusps that project to the cusp $\infty$ of $X_0(p)$, and these cusps are defined over the real subfield of the $p$-th cyclotomic field and thus form a single Galois orbit. As a consequence, it is easy to understand the cusps of $X_H(\ell N)$, including the Galois action on them, and so constructing $D_\infty$ poses no difficulty.

Since the level $\ell N$ is squarefree, the group spanned by the operators $W_Q$ and $\langle d \rangle$ acts transitively on the cusps, so the formulas established in section \ref{sect:WQ} allow us to compute the $q$-expansion of any modular form in $\M_2\big(\Gamma_H(\ell N)\big)$ at each of the cusp, given its $q$-expansion at the cusp $\infty$. In particular, it is easy to compute a basis of the 2-dimensional space $E$ by linear algebra. Therefore, we choose to represent the elements of $V_2$ by their $q$-expansion at all the cusps, with enough $q$-adic accuracy to ensure that Makdisi's algorithms perform correctly. More precisely, if we expand to accuracy $O(q^{B_c})$ at the cusp $c$, then it is enough that
\[\sum_{c \in D_\infty} B_c + \sum_{c \not \in D_\infty} (1 + B_c) >6 d_0 \]
since these algorithms deal with subspaces of $H^0\big(X_H(\ell N), n D_0 \big)$ for $n \in \N$ up to $6$.

\bigskip

Now that we are able to compute in the jacobian $J_H(\ell N)$ of $X_H(\ell N)$, we may proceed just as in \cite{algo}, by identifying the 2-dimensional subspace $V_{f,\l}$ of $J_H(\ell N)[\ell]$ that affords $\rho_{f,\l}$ thanks to the Fourier coefficients $a_n$ of $f$ for small $n$, inverting the Abel-Jacobi map at the points of $V_{f,\l}$, and evaluating a rational map $\alpha \in \Q\big( J_H(\ell N) \big)$ at these points, as in sections 3.5 and 3.6 of \cite{algo}. Finally, we identify the coefficients of  
\[ F(x) = \prod_{\substack{P \in V_{f,\l} \\ P \neq 0}} \big(x - \alpha(P) \big) \]
as rational numbers. If these identifications are correct and if $\alpha$ is one-to-one on $V_{f,\l}$, then the polynomial thus obtained describes the representation $\rho_{f,\l}$; in particular, its Galois group over $\Q$ is $\Im \rho_{f,\l}$.

\section{Certification of the results}\label{sect:certif}

We now wish to certify that the data computed in the previous section does define the representation $\rho_{f,\l}$. This is very likely, but not completely sure, as these data were produced by identifying rational numbers from floating point approximations.

The certification method we present here is a generalisation of the one presented in\cite{certif}. We still focus on the case where the level $N$ of $f$ is squarefree, although it is probably not difficult to drop this hypothesis, possibly at the expense of slowing down the computations. For simplicity, we also assume that the representation $\rho_{f,\l}$ is surjective; it is easy to modify our arguments when this is not the case.

In order to completely certify our data, we will eventually have to restrict to the case where the nebentypus $\eps$ of $f$ is trivial; unlike in section \ref{sect:algo}, this is a real requirement which the author does not know how to remove. It is however not necessary to make this assumption if one is only interested in the \emph{projective} representation attached to $f \bmod \l$, which is the case if one just wants to construct explicitly $\PGL_2(\F_\l)$-number fields with small discriminants.

\subsection{Reduction of the polynomials}

The polynomial $F(x)$ computed above tends to have a very large arithmetic height. More precisely, in \cite[section 2]{certif}, we predicted that the number of decimal digits of the typical denominator of the coefficients of $F(x)$ was approximately $g^{5/2}$, where $g$ is the genus of the modular curve used in the computation.

We have computed two representations (cf. the results section) mod $\ell=13$ in the jacobian of modular curves $X_H$ of respective levels $5 \cdot 13$ and $7 \cdot 13$ which both have genus $g=13$. The denominator of the polynomial $F(x)$ thus obtained has 458 decimal digits for the first representation, and 586 for the second one, which indicates that our prediction extends to the modular curves $X_H(\ell N)$ (as opposed to $X_1(\ell)$). This was expected, as our prediction is governed by the genus and not by the level.

\bigskip

Anyhow, it is extremely inconvenient to work with polynomials of such height, so we want to apply \cite{gp}'s function \verb?polredbest? to them, as this function computes a nicer polynomial defining the same number field. As noted in \cite{certif}, very often the polynomial $F(x)$ is simply too large for this to be possible; however, we may form the polynomials
\[ F_i(x) = \prod_{S \cdot P \in V_{f,\l}^{S_i}} \left( x - \sum_{s \in S_i} \alpha(s \cdot P) \right) \]
that ought to correspond to the quotient representations
\[ \xymatrix{\rho_{f,\l}^{S_i} : G_\Q \ar[r]^{\rho_{f,\l}} & \GLFl \ar@{->>}[r] & \GLFl / S_i} \]
for $0 \leqslant i \leqslant r$, and identify their coefficients as rationals, where $S_i = \{ s^{2^i}, \ s \in \F_\l^* \}$, $V_{f,\l}^S = (V_{f,\l} - \{0\}) / S$, and $r$ is the 2-adic valuation of $\# \F_\l^*$; we may then reduce these polynomials inductively on $i$ as explained in section 2 of \cite{certif}.

The point of this is that the polynomial $F_r(x)$ ought to correspond to the quotient representation $\rho_{f,\l}^{S_r}$, which contains enough information to recover $\rho_{f,\l}$ itself while being much easier to deal with, as explained in section 2 of \cite{certif}.

\subsection{Certification of the data}

Now that the polynomials have been reduced, we begin as in \cite{certif} by proving that the Galois group of $F_0(x)$ over $\Q$ is $\PGL_2(\F_\l)$, for instance thanks to the ``unordered cross-ratio'' method presented in section 3.3.1 of \cite{algo}. If we are only interested in the construction of $\PGL_2(\F_\l)$-number fields with small discriminant, we may stop here, check that the root field of $F_0(x)$ has as little ramification as expected, and add $F_0(x)$ to our collection; in fact, we did not need to compute and reduce the polynomials $F_i(x)$ for $i > 0$ in the first place.

However, if we are interested in the representation $\rho_{f,\l}$, then we need to certify that the polynomial $F_r(x)$ corresponds in the sense of \cite{certif} to the quotient representation $\rho_{f,\l}^{S_r}$. In order to do this, we now introduce a generalisation of the methods presented in \cite{certif}.

The reason why we need to modify these methods is that theorem 4 from \cite{certif}, which was used to certify the modularity of the projective representation defined by $F_0(x)$, only applies to representations attached to forms of level $1$, which is not the case in this article, and which are wildly ramified at $\ell$, which is precisely not the case we are most interested in. As a result, we present a new, more general method to certify that a projective Galois representation is modular and comes from a form of squarefree level.

\bigskip

We begin by recalling a well-known result about projective Galois representations.

\begin{lem}\label{lem:char_ext}
Let $\F$ be a topological field, $p \in \N$ a prime, and let $I_p \leqslant G_\Q$ be the inertia subgroup attached to some prime of $\overline \Q$ above $p$. Every continuous character $\chi : I_p \longrightarrow \F^*$ may be extended to a continuous character $G_\Q \longrightarrow \F^*$.  Similarly, every continuous character $\chi : W_p \longrightarrow \F^*$ may be extended to a continuous character $G_\Q \longrightarrow \F^*$, where $W_p$ is the wild inertia subgroup.
\end{lem}

%
%


\begin{thm}[Serre, Tate]\label{thm:lift_pi}
Let $\pi : G_\Q \longrightarrow \PGL_2(\overline \Fl)$ be a projective Galois representation. There exists a lift $\rho : G_\Q \longrightarrow \GL_2(\overline \Fl)$ such that for all primes $p \in \N$,
\[ \pi \text{ is unramified at p } \Longrightarrow \rho \text{ is unramified at p } \]
and
\[ \pi \text{ is tamely ramified at p } \Longrightarrow \rho \text{ is tamely ramified at p.} \]
\end{thm}

Note that the field of definition of $\rho$ may be larger than the one of $\pi$.

\bigskip

Thanks to this result, we may lift projective representations into linear ones, to which we may apply Serre's modularity conjecture so as to prove that the original projective representation is modular. This is very useful for us as we want to certify that our data correspond to modular Galois representations; however, this is not enough, as we want to prove that the representation corresponding to our data is attached to the eigenform $f$ and not another one. In this view, we establish the following theorem, that allows us to prove that our data define a representation attached to a form of the same level as $f$ with very little computational effort.

\begin{thm}\label{thm:Serre_proj}
Let $\pi : G_\Q \longrightarrow \PGL_2(\overline \Fl)$ be a projective Galois representation, let $K$ be the number field corresponding via $\pi$ to the stabiliser of a point of $\Pl{\overline \Fl}$. Let $R$ be the set of primes $p \neq \ell$ at which $\pi$ ramifies. Suppose that $\pi$ is irreducible and odd, and that for all $p \in R$, $\pi$ is tamely ramified at $p$, and there exists an unramified prime $\p$ of $K$ above $p$. Then there exists a newform $f \in \S_k\big( \Gamma_1(N) \big)$ with $N = \prod_{p \in R} p$ and $2 \leqslant k \leqslant \ell+2$ such that $\pi$ is equivalent to the projective representation $\pi_{f,\l}$ attached to $f$ modulo a prime $\l$ above $\ell$.
\end{thm}

\begin{proof}
By theorem \ref{thm:lift_pi}, there exists a lift $\rho : G_\Q \longrightarrow \GL_2(\overline \Fl)$ of $\pi$ which is irreducible, odd, unramified at the primes at which $\pi$ is unramified, and tamely ramified at all primes $p \neq \ell$. Therefore, Serre's modularity conjecture, which was proved by Khare and Wintenberger in \cite{KW}, implies that there exists a newform $f \in \S_k\big( \Gamma_1(N) \big)$ such that $\rho$ is equivalent to the representation attached to $f$ mod $\l$, where $N = \prod_{p \in R} p^{n_p}$ is the Serre conductor of $\rho$ and $\l$ is a prime above $\ell$. Besides, according to \cite[theorem 2.7]{RibetStein}, after twisting by a character of $\ell$-power conductor (which does not affect $\pi$), we may suppose that $2 \leqslant k \leqslant \ell+2$, so to conclude we only have to prove that $n_p = 1$ for all $p \in R$.

Let $p \in R$, so that by our hypothesis there exists an unramified prime $\p$ of $K$ above $p$. Let $\P$ be a prime of $\overline \Q$ above $\p$, and let $I_\P \leqslant D_\P \leqslant G_\Q$ be its inertia and decomposition subgroups. Since $\rho$ is tamely ramified at $p$, the local exponent of its conductor is just
\[ n_p = \codim V^{\rho(I_\P)}, \]
where $V \simeq \overline \Fl^2$ is the space of the representation. As $\pi$ ramifies at $p$, so does $\rho$, so $\rho(I_p)$ is not trivial and $n_p \geqslant 1$. Let $L$ be the Galois number field cut out by $\rho$, and let $B = \Gal(L/K)$, so that $\rho(B)$ is contained in a Borel subgroup of $\GL_2(\overline \Fl)$ by definition of $K$. As $\p$ is unramified, we have a tower of local extensions
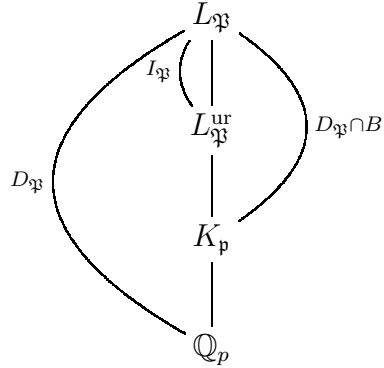
\[ \xymatrix{
L_\P \ar@{-}[d] \ar@/_5pc/@{-}[ddd]_{D_\P} \ar@/_1pc/@{-}[d]_{I_\P} \ar@/^3pc/@{-}[dd]^{D_\P \cap B} \\
L_\P^\text{ur} \ar@{-}[d] \\
K_\p \ar@{-}[d] \\
\Q_p
} \]
where $L_\P^\text{ur}$ is the maximal unramified subextension of $L_\P$. Therefore, $I_\P$ is contained in $B$, so $\rho \vert_{I_\P} \sim \smat{\chi_p}{*}{0}{*}$ for some character $\chi_p : I_\P \longrightarrow \overline \Fl^*$, which we may extend to $G_\Q$ by lemma $\ref{lem:char_ext}$. After replacing $\rho$ with $\rho \otimes \prod_{p \in R} \chi_p^{-1}$, which does not affect the weight since $\ell \not \in R$, we may thus suppose that $\rho \vert_{I_\P} \sim \smat{1}{*}{0}{*}$, so that $n_p \leqslant 1$. This concludes the proof.
\end{proof}

\begin{rk}
Conversely, let $\rho = \rho_{f,\l}$ be a mod $\ell$ representation attached to an eigenform $f \in \S_k\big( \Gamma_1(N) \big)$. We may suppose that $N$ is minimal, that is to say that it is the Serre conductor $\prod_{p \neq \ell} p^{n_p}$ of $\rho$, where
\begin{equation} n_p = \codim V^{\rho(I_p)} + \sum_{n \geqslant 1} \frac1{[I_p : I_p^{(n)}]} \codim V^{\rho(I_p^{(n)})} \label{local_cond} \end{equation}
and the $I_p^{(n)} \leqslant W_p$ are the higher ramification groups for $n \geqslant 1$. Then if $p \in \N$ is a prime such that $p \parallel N$ so that $n_p=1$, the equation \eqref{local_cond} implies that $\rho$ is tamely ramified at $p$ and that the inertia at $p$ fixes a dimension 1 subspace of $V$, so that the number field $K$ corresponding by $\rho$ to the stabiliser of a point in $\Pl{\F_\l}$ has an unramified prime $\p$ above $p$. Therefore, the implication between the existence of an unramified prime $\p$ and the fact that $n_p=1$ is actually an equivalence, so that theorem \ref{thm:Serre_proj} yields a very efficient and general-purpose criterion that we can use to prove that a polynomial with Galois group a subgroup of $\PGL_2(\F_{\ell^m})$ defines a projective representation which is modular of squarefree level.
\end{rk}

Thanks to this criterion, we are able to prove that the polynomial $F_0(x)$ that we have computed defines a projective representation $\pi_F$ attached to an eigenform $f'$ of the same level $N$ as $f$ and of weight $2\leqslant k \leqslant \ell+2$ modulo a prime $\l'$. We now want to ensure that this representation is actually attached to $f$ modulo the prime $\l$. We will prove this by listing all the possible candidate forms, and eliminating them one by one. Of course, the fact that we have already determined the level narrows down this search considerably.

To do so, we apply a generalisation of the technique presented in the second half of section 3.3.2 of \cite{certif}: for each prime $p$ such that
\begin{equation} \begin{array}{c} F_0(x) \bmod p \text { is squarefree and splits as a product of} \\ \text{linear or quadratic factors, but does not split completely,} \label{cond_p_invol} \end{array} \end{equation}
we know that $\pi_F(\Frob_p)$ is of order exactly $2$, so that its trace is zero. As a result, any eigenform $f'$ such that $\pi_{F} \sim \pi_{f',\l'}$ for some $\l'$ must satisfy $a_p(f') \equiv 0 \bmod \l'$.

We thus form the list of couples $(f',\l')$, where $f'$ is a newform of level $\Gamma_1(N)$ and weight\footnote{We could also have used the range $1 \leqslant k \leqslant \ell+1$ throughout this section; we prefer the interval $2 \leqslant k \leqslant \ell+2$ because cuspforms of weight $1$ are difficult to compute, so that using the range $1 \leqslant k \leqslant \ell+1$ would make the generation of the list of the $(f',\l')$ unnecessarily delicate.} between $2$ and $\ell+2$, and $\l'$ is a prime of the Hecke field of $f'$ of the appropriate degree above $\ell$; then we start looking for primes $p$ satisfying the condition \eqref{cond_p_invol}, and for each such prime we eliminate the couples $(f',\l')$ that fail to satisfy the condition $a_p(f') \equiv 0 \bmod \l'$. This is very efficient, as each such prime $p$ divides the size of the list roughly by $\# \F_\l$. Besides, in order to speed up the computation, we can replace the condition $a_p(f') \equiv 0 \bmod \l'$ by $N^{K_f'}_\Q\big(a_p(f')\big) \equiv 0 \bmod \ell$ where $N^{K_f'}_\Q$ is the norm from the Hecke field of $f'$ to $\Q$, which is weaker but just as discriminating in practice, and allows us to barely have to deal with the different possible primes $\l'$ above $\ell$ at all.

We stop when all the remaining couples $(f',\l')$ correspond to the same projective representation. In general, this happens when the only couple left on the list is $(f,\l)$ itself, except of course when $f$ admits a companion mod $\ell$, in which case we wait for the list to reduce to the couple $(f,\l)$ and the companion couple. Typically, it is enough to consider the primes $p \leqslant 100$ to achieve this.

\bigskip

We are thus able to certify that $F_0(x)$ defines the projective representation attached to $f \bmod \l$. If the nebentypus $\eps$ of $f$ is trivial, we may then apply without any modification the ``group cohomology method'' presented in section 3.6 of \cite{certif} to certify that the polynomial $F_r(x)$ defines the quotient representation  $\rho_{f,\l}^{S_r}$ attached to $f \bmod \l$, and to compute the image of Frobenius elements in a certified way. Unfortunately, the author does not know at present how to do the same thing if $\eps$ is not trivial.

\newpage

\section{Results}\label{sect:results}

We have used the above algorithms to compute the mod $13$ Galois representations attached to the primitive newforms
\[ q + 2q^2 - 4q^3 +O(q^4) \in \newf_6\big( \Gamma_0(5) \big) \]
and
\[ q - 6q^2 - 42q^3 + O(q^4) \in \newf_8\big( \Gamma_0(7) \big) \]
of respective LMFDB labels \mfref{5.6.1.a} and \mfref{7.8.1.a}. The former is supersingular mod 13, whereas the latter admits \mfref{7.6.1.a} as a companion mod $13$. The reason for the choice of these forms is that the corresponding Galois representations are afforded in the torsion of Jacobian of modular curves $X_H(\ell N)$ whose genus is moderate, namely $g=13$ in both cases; in particular, working with \mfref{7.6.1.a} instead of \mfref{7.8.1.a} would have led to computing the same projective representation but a different linear representation in a different curve of higher genus. Similarly, the projective representation attached to \mfref{5.6.1.a} also comes from \mfref{5.10.1.a} mod 13, but using the former leads to a modular curve of lower genus than with the latter.

We are currently computing the mod $41$ representation attached to the form
\[ q + 1728 q^2 - 59049 q^3 + O(q^4) \in \newf_{22} \big( \Gamma_0(3) \big)  \]
of LMFDB label \mfref{3.22.1.b} that admits \mfref{3.20.1.b} as a companion; the genus of $X_H(3 \cdot 41)$ is $g=25$ in this case, which is to the very top of the range of genera that are reasonably amenable to computation with our method.

In both cases, the computation of the Galois representations mod 13 took about 12 hours, after which the reduction of the polynomials by the inductive method took just a few minutes (as a comparison, the direct reduction of the polynomial $F_r(x)$ takes about 90 hours), and finally the whole certification process took less than a minute.

\subsection{The polynomials for the projective representations}

Recall that under GRH, we have
\[ \liminf_{n \rightarrow \infty} \inf_{\substack{[K : \Q] = n \\ \sign(K)=(r_1,r_2)}} \vert \disc K \vert^{1/n} \geqslant 8 \pi e^{\gamma + \frac\pi2 \frac{r_1}n} = 44.763 \dots  \times (4.810 \cdots)^{r_1/n}, \] 
cf. \cite{SerreDisc}.

We have computed and certified that the field corresponding to the stabiliser of a point of $\Pl{\F_{13}}$ via the projective representation attached to \mfref{5.6.1.a} mod 13 is defined\footnote{Polynomials defining a fixed number field are of course never unique; however, the polynomials that we display in this section are the output of \cite{gp}'s function \verb?polredabs?, which makes them canonical.} by the polynomial
\begin{multline*}
x^{14} - x^{13} - 26x^{11} + 39x^{10} + 104x^{9} - 299x^{8} - 195x^{7} \\ + 676x^{6} + 481x^{5} - 156x^{4} - 39x^{3} + 65x^2 - 14x + 1.
\end{multline*}
This polynomial has thus signature $(2,6)$ and a Galois group that is permutation-isomorphic to $\PGL_2(\F_{13}) \circlearrowleft \Pl{\F_{13}}$. Besides, the root discriminant of its stem field is
\[ (5^{12}13^{13})^{1/14} = 43.002 \dots, \]
which beats the smallest (as of November 2016) root discriminant for a $\PGL_2(\F_{13})$ field contained in the database \cite{KlunersMalle} which is
\[ (2^{14} 3^{12} 13^{13})^{1/14} = 55.509 \cdots. \]
Even better, it is significantly lower than the estimation
 \[ 8 \pi e^{\gamma+\pi/14} = 56.024 \cdots \]
for a field of this signature, and even lower than the estimation
\[ 8 \pi e^\gamma = 44.763 \cdots \]
for any number field of any signature. Nevertheless, our field is beaten by the record (as of November 2016) from the \cite{LMFDB}, namely the $\PGL_2(\F_{13})$-field \href{http://lmfdb.warwick.ac.uk/NumberField/14.2.20325604337285010030592.1}{14.2.20325604337285010030592.1} computed by Noam Elkies, whose root discriminant is only
\[ (2^{26}13^{13})^{1/14} = 39.213 \cdots. \] 

However, if we move to Galois closures, then our field sharply beats both the aforementioned field from \cite{KlunersMalle} and Elkies's. Indeed, the root discriminant of the Galois closure of our field is  
\[ 5^{12/13} 13^{13/14} = 47.816 \cdots, \]
which incidentally is close to $8 \pi e^{\gamma}$, whereas it is
\[ 2^{13/6} 13^{167/156} = 69.939 \cdots \]
for Elkies's field, and
\[ 2^{7/6} 3^{12/13} 13^{167/156} = 96.407 \cdots \]
for the above field from \cite{KlunersMalle}, and even worse for the other $\PGL_2(\F_{13})$-fields from \cite{KlunersMalle} and the \cite{LMFDB}. This is due to the fact that our field is tamely ramified at all primes, whereas the others aren't. According to \cite[p. 14]{Roberts_companion}, it is possible that our field is the $\PGL_2(\F_{13})$ field whose Galois closure has the smallest root discriminant.

\vspace{1cm}

Similarly, we have computed and certified that the polynomial corresponding to the projective representation attached to \mfref{7.8.1.a} mod 13 is
\[ x^{14} - 52x^7 + 91x^6 + 273x^5 - 364x^4 - 1456x^3 - 455x^2 + 1568x + 1495. \]
This polynomial has the unexpected property that all 6 terms from $x^{13}$ to $x^8$ included are missing. The root discriminant of its root field is
\[ (7^{12} 13^{11})^{1/14} = 39.775 \cdots \]
which is even better than our previous example and narrowly misses beating Elkies's; however the root discriminant of its Galois closure is
\[ 7^{12/13} 13^{11/12} = 63.271 \cdots, \]
which is a little less good than our previous example but still beats Elkies's, \cite{KlunersMalle} and the \cite{LMFDB}, again thanks to the fact that it is tamely ramified.

\subsection{The polynomials for the quotient representations}

For both of these representations mod $13$, we have also computed the polynomials $F_r(x)$ introduced in section \ref{sect:certif}, and certified that these polynomials are correct thanks to the group cohomology method presented in \cite{certif}. We have then computed the Dokchitsers' resolvents, that may be used to determine the image in $\GL_2(\F_{13})$ (up to similarity of course) of Frobenius elements, and in particulat to recover the value mod $13$ of the coefficients $a-P$ of these forms for huge primes $p$. All these data are available for download on the author's web page located at \url{https://www2.warwick.ac.uk/fac/sci/maths/people/staff/mascot/galreps/}.

\bigskip

For the representation attached to \mfref{5.6.1.a} mod 13, the polynomial $F_r(x)$ is
\begin{tiny}
\[ \hspace{-1cm} \arraycolsep=1.4pt \begin{array}{ccl} F_2(x) & = &
x^{56} - 19 \, x^{55} + 176 \, x^{54} - 1099 \, x^{53} + 5292 \, x^{52} - 19916 \, x^{51} + 53755 \, x^{50} - 82979 \, x^{49} - 11609 \, x^{48} + 418938 \, x^{47} - 1351519 \, x^{46} + 3570307 \, x^{45} \\
&& - 8104499 \, x^{44} + 9946931 \, x^{43} + 5331934 \, x^{42} - 12684220 \, x^{41} - 180933386 \, x^{40} + 956990587 \, x^{39} - 2345057533 \, x^{38} + 2930653050 \, x^{37} \\
&& - 366740868 \, x^{36} - 2647967569 \, x^{35} - 10686690040 \, x^{34} + 66782657110 \, x^{33} - 169078436150 \, x^{32} + 261459165916 \, x^{31} - 253975820897 \, x^{30} \\
&& + 159187764447 \, x^{29} - 272743393068 \, x^{28} + 1165595337221 \, x^{27} - 3256037467741 \, x^{26} + 6113796826345 \, x^{25} - 8131597368544 \, x^{24} \\
&& + 7180532683571 \, x^{23} - 2160263809470 \, x^{22} - 5641397045687 \, x^{21} + 12758000383973 \, x^{20} - 15558252071934 \, x^{19} + 12690172501916 \, x^{18} \\
&& - 6215260751330 \, x^{17} + 180457670019 \, x^{16} + 2797189991937 \, x^{15} - 3474577634674 \, x^{14} + 4227913001201 \, x^{13} - 5838445844387 \, x^{12} \\
&& + 6919193824400 \, x^{11} - 5805277968711 \, x^{10} + 2648204866489 \, x^{9} + 369252764894 \, x^{8} - 1933374840137 \, x^{7} + 1819874305834 \, x^{6} \\
&& - 1245647904878 \, x^{5} + 908803702639 \, x^{4} - 675346876626 \, x^{3} + 345380525276 \, x^{2} - 96857560911 \, x + 7979838361.
\end{array} \]\end{tiny}

For the one attached to \mfref{7.8.1.a} mod 13, it is
\begin{tiny}
\[ \hspace{-1cm} \arraycolsep=1.4pt \begin{array}{ccl} F_2(x) & = &
x^{56} - 14 \, x^{55} + 69 \, x^{54} - 82 \, x^{53} - 396 \, x^{52} + 823 \, x^{51} + 3351 \, x^{50} - 11931 \, x^{49} + 8522 \, x^{48} - 35835 \, x^{47} + 186446 \, x^{46} - 8847 \, x^{45} - 854460 \, x^{44} \\
&& - 743676 \, x^{43} + 4590031 \, x^{42} + 7212191 \, x^{41} - 22038546 \, x^{40} - 38957922 \, x^{39} + 49157879 \, x^{38} + 243902411 \, x^{37} - 180717704 \, x^{36} \\
&& - 988889224 \, x^{35} + 704374598 \, x^{34} + 4859375083 \, x^{33} - 3763415241 \, x^{32} - 16386779936 \, x^{31} + 21701597191 \, x^{30} + 46834006724 \, x^{29} \\
&& - 85332561468 \, x^{28} - 70138311949 \, x^{27} + 302231735974 \, x^{26} - 10052385427 \, x^{25} - 632464301217 \, x^{24} + 556951211889 \, x^{23} + 1081393994453 \, x^{22} \\
&& - 1845293759824 \, x^{21} - 358646925616 \, x^{20} + 3673731123829 \, x^{19} - 1686600977427 \, x^{18} - 4103844332008 \, x^{17} + 5303152415742 \, x^{16} \\
&& + 2644700341946 \, x^{15} - 8175629100848 \, x^{14} + 3069957630241 \, x^{13} + 5747922498716 \, x^{12} - 4528557603372 \, x^{11} - 3341692089599 \, x^{10} \\
&& + 6603867688269 \, x^{9} - 2399016765221 \, x^{8} - 1314878616927 \, x^{7} + 1252052945123 \, x^{6} - 5862989822 \, x^{5} - 159810157800 \, x^{4} - 23334202447 \, x^{3} \\
&& + 35386045540 \, x^{2} + 9146004182 \, x + 973774019.
\end{array} \]\end{tiny}

\section{What to expect in general}\label{sect:what_to_expect}

\subsection{Computation of the root discriminant}

We now want to establish a formula that will allow us to predict the value of the root discriminant of the fields that we obtain when we compute the projective representation attached to a form of squarefree level and which admits a companion or is supersingular.

We begin by recalling some well-known relations between the discriminant of a number field, the ramification and splitting behaviour of primes in this field, and the Galois action.

\begin{lem}\label{lem:primes}
Let $T(x) \in \Q[x]$ be an irreducible polynomial of degree $n$, \linebreak $Z = \{ z_1, \cdots, z_n \}$ be the set formed by its roots in $\overline \Q$, $K = \Q(z_1)$ be the associated number field of degree $n$, $L=\Q(z_1,\cdots,z_n)$ its Galois closure, and $G = \Gal(L/\Q)$ its Galois group. Fix a prime $p \in \N$ and a prime $\P$ of $L$ above $p$, and let $I_\P \trianglelefteq D_\P \leqslant G$ be the inertia and decomposition groups of $\P$. Then the map that associates to $g \in G$ the prime $(g \cdot \P) \cap K$ of $K$ induces a bijection between the orbits of $D_\P$ acting on $Z$ and the primes $\p$ of $K$ above $p$. Furthermore, given such an orbit $\omega$, the inertial degree $f_{\p/p}$ of the corresponding prime $\p$ is the number of orbits of $\omega$ under $I_\P$, and these orbits all have the same size, which agrees with the ramification index $e_{\p/p}$ of $\p$. Finally, if the ramification is tame at $p$, that is to say if the indices $e_{\p/p}$ are coprime to $p$ for all $\p \mid p$, then the $p$-adic valuation of the root discriminant of $K$ is
\[ \alpha_p = 1-\frac1n \sum_{\p \mid p} f_{\p/p} = 1 - \frac1n \# I_\P \backslash Z, \]
and the one of the root discriminant of $L$ is
\[ \beta_p = 1-\frac1{\lcm_{\p \mid p} e_{\p / p}} = 1-\frac1{\lcm_{\omega \in I_\P \backslash Z} \# \omega}. \] 
\end{lem}

Examining the action of inertia through a modular Galois representation then leads to the following formulas:

\begin{thm}\label{thm:predict_disc}
Let $\pi_{f,\l} \colon G_\Q \longrightarrow \PGL_2(\F_\l)$ be the projective representation attached to a newform $f = q + \sum_{n \geqslant 2} a_n q^n \in \newf_k(N,\eps)$ of weight $2 \leqslant k \leqslant \ell+1$, where $\l$ is a prime of the Hecke field of $f$ above an odd prime $\ell$. Suppose that the image of $\pi_{f,\l}$ is not too small, in that it acts transitively on $\PP^1(\F_\l)$. Let $m$ be the degree of $\ell$, let $K$ be the number field of degree $\ell^m+1$ corresponding via $\pi_{f,\l}$ to the stabiliser of a point of $\PP^1(\F_\l)$, and let $L$ be its Galois closure, which is thus the field cut out by $\pi_{f,\l}$. Assume that $N$ is squarefree and that $\ell \nmid N$. Let $M$ be the conductor of $\eps \bmod \l$, and for each prime $p \mid M$, let $r_p \in \N$ denote the multiplicative order of the $p$-part of $\eps \bmod \l$. Finally, let $N'$ be the product of the primes $p \neq \ell$ such that the \emph{linear} representation attached to $f \bmod \l$ is ramified at $p$, so that $M \mid N' \mid N$. Then the root discriminant of $K$ is
\[ d_K = \ell^\alpha \left( \frac{N'}{M} \right)^{\frac{1-1/\ell}{1+1/\ell^m}} \left( \prod_{p \mid M} p^{1-1/r_p} \right)^{\frac{\ell^m-1}{\ell^m+1}} \]
for some $\alpha \in \Q_{>0}$, whereas the root discriminant of $L$ is
\[ d_L = \ell^\beta \left( \frac{N'}{M} \right)^{1-1/\ell}  \prod_{p \mid M} p^{1-1/r_p} \]
for some $\beta \in \Q_{>0}$.
Furthermore, if the coefficient $a_\ell$ of $f$ is not $0 \bmod \l$ and if $f \bmod \l$ admits a companion form, then $\pi_{f,\l}$ is tamely ramified at $\ell$ and we have
\[ \beta = 1-\frac{\gcd(k-1,\ell-1)}{\ell-1}, \quad \alpha = \frac{\ell^m-1}{\ell^m+1}\beta, \]
whereas if the coefficient $a_\ell$ of $f$ is $0 \bmod \l$, then $\pi_{f,\l}$ is again tamely ramified at $\ell$ and we have
\[ \beta = 1-\frac{\gcd(k-1,\ell+1)}{\ell+1}, \quad  \alpha = \left\{ \begin{array}{ll} \beta & \text{ if } m \text{ is odd,} \\ \frac{\ell^m-1}{\ell^m+1} \beta & \text{ if } m \text{ is even.} \end{array} \right. \]
\end{thm}

\begin{rk}
The value of $N'$ is not difficult to determine in practice. Indeed, a prime $p$ divides $N/N'$ if and only if there exists an eigenform $f'$ of level $N/p$ such that $f' \bmod \l' = f \bmod \l$ for some prime $\l'$ above $\ell$; as a consequence, $N'$ is the product of the primes $p$ such that no such form $f'$ exists.
\end{rk}

\begin{rk}
The case when the ramification at $\ell$ is wild is treated in \cite{MoonTaguchi}.
\end{rk}

\begin{rk}
The image of complex conjugation by $\pi_{f,\l}$ is similar to $\smat{1}{}{}{-1}$ as $\ell$ is odd, so the signature of $K$ is $\left( 2,\frac{\ell^m-1}2 \right)$ and the one of $L$ is $(0, \frac{\# \Im \pi_{f,\l}}2)$. This allows us to determine the sign of the discriminants of $K$ and $L$, should we want to do so.
\end{rk}

\begin{proof}
Let $\rho = \rho_{f,\l} : G_\Q \longrightarrow \GL_2(\F_\l)$ be the linear representation attached to $f \bmod \l$, and write $d_K$ and $d_L$ for the root discriminants of $K$ and $L$.

Let $p \neq \ell$ be a prime dividing $N'$. By hypothesis, $p \parallel N$ and $\rho$ ramifies at $p$, so $\rho(I_p)$ fixes a line point wise by \eqref{local_cond}. We thus have
\[ \rho \vert_{I_p} \sim \mat{1}{\xi}{}{\eps_p}, \]
where $\eps_p = \det \rho \vert_{I_p}$ is the $p$-part of $\eps \bmod \l$. We now distinguish two cases.

On the one hand, if $\eps_p$ is non-trivial, that is to say if $p \mid M$, then by section 2 of \cite{Diamond} we may arrange that $\xi=0$, so $I_p$ acts on $\PP^1(\F_\l)$ via $\eps_p$, whence two orbits formed of a single point (namely $0$ and $\infty$), and $\frac{\ell^m-1}{r_p}$ orbits of size $r_p$.

On the other hand, if $\eps_p$ is trivial, that is to say if $p \mid \frac{N'}M$, then $\xi$ is an additive character on $I_p$ with values in $\F_\l$. As $p \neq \ell$, $\xi$ factors through the tame quotient of $I_p$; since this quotient is cyclic, the image of $\xi$ is either cyclic of order $\ell$ or trivial. But $\xi$ cannot be trivial as $\rho$ ramifies at $p$; therefore, $I_p$ acts on $\PP^1(\F_\l)$ by a cyclic group of translations of order $\ell$, whence one orbit of size $1$ (the point at infinity) and $\ell^{m-1}$ orbits of size $\ell$.

Either way, the ramification is tame at $p$ (which we already knew by the formula \eqref{local_cond} for the Artin conductor), so we can compute the exponent of $p$ in $d_K$ and in $d_L$ thanks to lemma \ref{lem:primes}.

We now focus on the image of the inertia at $\ell$. We distinguish again two cases.

On the one hand, if $f \bmod \l$ is \emph{ordinary}, that is to say if $a_\ell \not \equiv 0 \bmod \l$, then by \cite[proposition 12.1]{Gross}, we have
\[ \rho \vert_{D_\ell} \sim \mat{\chi_\ell^{k-1} \alpha \ }{\xi}{}{\beta}, \]
where $\chi_\ell$ is the mod $\ell$ cyclotomic character (the one that tells the action of Galois on the $\ell$-th roots of unity), and $\alpha$ and $\beta$ are unramified characters. Besides, $\xi = 0$ by \cite{Gross} if $f \bmod \l$ admits a companion form. Therefore,
\[ \rho \vert_{I_\ell} \sim \mat{\chi_\ell^{k-1}}{}{}{1}. \]
Therefore, $I_p$ acts on $\PP^1(\F_\l)$ by multiplication by the $(k-1)$-th powers of $\Flx$, so we have two orbits of size 1 and $\frac{\ell^m-1}{\gcd(k-1,\ell-1)}$ orbits of size $\gcd(k-1,\ell-1)$.

On the other hand, if $f \bmod \l$ is \emph{supersingular}, that is to say if $a_\ell \equiv 0 \bmod \l$, then according to section 2.1.2 of \cite{RibetStein}, we have
\[ \rho \vert_{I_\ell} \sim \mat{\psi^{k-1}}{}{}{\psi'^{k-1}} \]
over $\overline \F_\l$, where $\psi$ and $\psi' = \psi^{\ell}$ are the two fundamental characters of level $2$, which take values in $\F_{\ell^2}^*$.

If $m$ is odd, then this splitting does not occur over $\F_\l$, so $\rho \vert_{I_\ell}$ is similar to $m$ copies of
\begin{equation} \begin{array}{rcccc} \F_{\ell^2}^* & \longrightarrow & \Aut_{\Fl}(\F_{\ell^2}) & \simeq & \GLFl \\ x & \longmapsto & (y \mapsto x y). & & \end{array} \label{psi_nonsplit} \end{equation}
restricted to the $(k-1)$-th powers of $\F_{\ell^2}^*$. Here, we are using the fact that $\F_{\ell^2}^*$ is cyclic and that two matrices with coefficients in $\F_\l$ which are similar over $\overline \F_\l$ are already similar over $\F_\l$. In \eqref{psi_nonsplit}, the $(\ell+1)$-th powers of $\F_{\ell^2}^*$ act as scalars since they lie in $\Flx$, so the action of $I_\ell$ on $\PP^1(\Fl)$ is equivalent to the action of the $(k-1)$-th powers of copies of an $(\ell+1)$-cycle. We thus have $\frac{\ell^m+1}{(\ell+1)/\gcd(k-1,\ell+1)}$ orbits, all of size $\frac{\ell+1}{\gcd(k-1,\ell+1)}$.

If $m$ is even, then we have the decomposition
\[ \rho \vert_{I_\ell} \sim \psi^{k-1} \otimes \mat{1}{}{}{\psi^{(\ell-1)(k-1)}} \]
over $\F_\l$. The character $\psi^{(\ell-1)(k-1)}$ is of order $r=\frac{\ell^2-1}{\gcd\big( (\ell-1)(k-1),\ell^2-1\big)} = \frac{\ell+1}{\gcd(k-1,\ell+1)}$, and the action of $I_\ell$ on $\PP^1(\F_\l)$ yields two orbits of size 1 and $\frac{\ell^m-1}{r}$ orbits of size $r$.

Either way, the ramification is tame, and we can again determine the valuation of $\ell$ in $d_K$ and $d_L$ thanks to by lemma \ref{lem:primes}.
\end{proof}

\begin{rk}One sees easily that same formulas remain valid for $k=1$, in which case one has $\alpha=\beta=0$. Furthermore, the same reasoning can also be used to derive formulas for the discriminant of the fields attached to the \emph{linear} representation $\rho_{f, \l}$ attached to a newform of any weight $k \leqslant \ell+1$ if desired.\end{rk}

\bigskip

Of course, we are mainly interested in the cases when the image of $\pi_{f,\l}$ is $\PGL_2(\Fl)$ or $\PSL_2(\Fl)$. Given an explicit choice of $f \in \newf_k(N,\eps)$ and $\l$, it is easy to ensure this. Indeed, if it were not the case, then according to Dixon's classification of finite subgroups of $\PGL_2(\overline{\Fl})$, either the image of the linear representation $\rho_{f,\l}$ attached to $f \bmod \l$ would be contained in a Borel subgroup or in the normaliser of a split Cartan subgroup in $\GL_2(\Fl)$, or the image of $\pi_{f,\l}$ would be isomorphic to a subgroup of the symmetric group $\mathfrak{S}_4$ or of the alternating group $\mathfrak{A}_5$. In order to rule these cases out, we use the same technique as in section 2 of \cite{SwD}: The Borel (resp. normaliser of Cartan) case can be ruled out by computing the Fourier coefficients $a_p$ of $f$ for a few primes $p$, and finding at least one $p \nmid \ell N$ such that $x^2-a_p x+p^{k-1}$ is irreducible in $\F_\l[x]$ (resp. a few primes $p \nmid \ell N$ which span $(\Z/\ell N \Z)^* \otimes \Z/2\Z$ and such that $a_p \not \equiv 0 \bmod \l$). Similarly, if $\ell \geqslant 7$, the $\mathfrak{S}_4$ or $\mathfrak{A}_5$ case may be ruled out by exhibiting a prime $p \neq \ell$ such that $a_p^2/p^{k-1} \bmod \l$ is not a root of $x(x-1)(x-2)(x-4)(x^2-3x+1)$. One then sees if the image of $\pi_{f,\l}$ is $\PSL_2(\Fl)$ or $\PGL_2(\Fl)$ by checking whether the values of the mod $\ell N$ character $x \mapsto x^{k-1} \eps(x) \bmod \l$ are all squares in $\F_l^*$ or not. 

\begin{ex}
For instance, we can easily check this way that the projective representation attached to \mfref{3.22.1.b} mod $41$ is surjective; since this form is ordinary at 41, the corresponding field $K$ of degree 42 that we are currently computing has root discriminant
\[ d_K = (3^{40} 41^{39})^{1/42} = 89.533 \cdots, \]
whereas its Galois closure has root discriminant
\[ d_L = 3^{40/41} 41^{41/42} = 109.614 \cdots. \]
\end{ex}

\subsection{A few examples with trivial nebentypus}

In the case where $f \bmod \l$ has trivial nebentypus, is ordinary, and admits a companion, we get especially small root discriminants when $N'$ is small and $k-1$ and $\ell-1$ have a large common factor. If they do not, for large $\ell$ the root discriminant of both $K$ and $L$ is asymptotically $\ell N$. This condition should not be confused with the condition from section \ref{sect:XH} that $\gcd(k-2,\ell-1)$ be large for $X_H(\ell N)$ to be much smaller than $X_1(\ell N)$. Sadly, these two conditions are rather contradictory\footnote{Unless $k=2$ of course, since we do not have to raise the level from $N$ to $\ell N$ in this case.}, so unfortunately the most interesting examples are the hardest to compute. For instance, of all the examples of newforms with trivial nebentypus, rational coefficients, and which admit a companion listed by David Roberts in table 3.2 of \cite{Roberts_companion}, the only case for which we are able to compute the associated Galois representation is for \mfref{7.8.1.a} mod 13, which we presented in section \ref{sect:results}.  

\bigskip

When $f \bmod \l$ is supersingular and has trivial nebentypus, the condition for the root discriminant to drop  from the $\ell N$ asymptotic is that $k-1$ and $\ell+1$ have a large common factor, and this seems more compatible with the condition for $X_H(\ell N)$ to have moderate genus. However, for $11 \leqslant \ell \leqslant 41$ we have found no examples of forms of weight $2 \leqslant k \leqslant \ell+1$, squarefree level $N \leqslant 20$ corpime to $\ell$ and trivial nebentypus which are supersingular mod~ a prime above $\ell$, whose mod $\ell$ projective representation has big image, and such that $\gcd(k-1,\ell+1)>1$. This is partly because having a trivial nebentypus forces $k-1$ to be odd, which makes it harder for its gcd with the even number $\ell+1$ to be nontrivial. We have still found two examples of surjective representations attached to supersingular forms leading to small root discriminants, even though $\gcd(k-1,\ell+1)=1$ in each case:

\begin{itemize}

\item The mod $\ell=19$ projective representation attached to
\[ \mfref{3.10.1.b} = q+18q^2+81q^3+O(q^4) \in \newf_{10}\big(\Gamma_0(3)\big) \]
is surjective and this form is supersingular mod $19$, whence an example of $\PGL_2(\F_{19})$-field with
\[ d_K = (3^{18} 19^{19})^{1/20} = 44.078 \cdots \]
and
\[ d_L = 3^{18/19} 19^{19/20} = 46.432 \cdots. \]
This projective representation also comes from the newform
\[ \mfref{3.12.1.a} = q+78q^2-243q^3+O(q^4) \in \newf_{12}\big(\Gamma_0(3)\big). \]

\item The mod $\ell=29$ projective representation attached to
\[ \mfref{2.14.1.a} = q-64q^2-1836q^3+O(q^4) \in \newf_{14}\big(\Gamma_0(2)\big) \]
is surjective and this form is supersingular mod $29$, whence an example of $\PGL_2(\F_{29})$-field with
\[ d_K = (2^{28} 29^{29})^{1/30} = 49.500 \cdots \]
and
\[ d_L = 2^{28/29} 29^{29/30} = 50.617 \cdots. \]
This projective representation also comes from the newform
\[ \mfref{2.18.1.a} = q+256q^2+6084q^3+O(q^4) \in \newf_{18}\big(\Gamma_0(2)\big). \]

\end{itemize}

The corresponding degree $\ell+1$ fields are not part of \cite{KlunersMalle} nor of the \cite{LMFDB}; according to \cite[p.14]{Roberts_companion}, it is possible that these fields are the ones of smallest root discriminant with this Galois group, and similarly for their Galois closures. Unfortunately, the genus of the modular curve $X_H(\ell N)$ is respectively $g=43$ and $g=36$ in these examples (no matter which of the two possible newforms we look at), which keeps them out of computational reach at present.

\subsection{Supersingular forms with nontrivial nebentypus}

As supersingularity is very easy to test, we have run a search form newforms of squarefree, coprime to $\ell$ level at most $20$ for $7 \leqslant \ell \leqslant 13$ and at most $10$ for $17 \leqslant \ell \leqslant 41$ which have (possibly odd) weight\footnote{According to theorem 2.8 of \cite{Edi92}, this is equivalent to searching for weights between $2$ and $\ell+1$.} $2 \leqslant k \leqslant \ell$, are supersingular modulo at least one prime $\l$ (of any degree) above $\ell$, and have nontrivial nebentypus, and we have kept the cases for which $d_K \leqslant 8 \pi e^{\gamma}$ or $d_L \leqslant 70$. Here, as before $d_K$ is the root discriminant of the field corresponding to the stabiliser of a point of $\PP^1(\F_\l)$ and $d_L$ is the root discriminant of its Galois closure.

In most cases, imposing these tight bounds on the root discriminants leads to representations with small image\footnote{A similar phenomenon is reported in section 4.5 of \cite{Roberts_companion}.}, the most frequent case being that the image is contained in the normaliser of a Cartan subgroup. After eliminating these cases, we are left with only two projective representations whose image is either $\PGL_2(\F_\l)$ or $\PSL_2(\F_\l)$:

\rowcolors{2}{gray!25}{white}
\[
\begin{array}{|c|c|c|c|c|}
\hline
\rowcolor{gray!50}
\text{Galois group} & d_K & d_L & \text{Newforms} & \text{Genus} \\
\hline
\PGL_2(\F_7) & 27.269 \cdots & 46.531 \cdots & \mfref{13.2.2.a}, \ \mfref{13.8.4.a} & 2 \\
\PSL_2(\F_{37}) & 51.483 \cdots & 52.993 \cdots & \mfref{3.13.2.b}, \ \mfref{3.27.2.a} & 385 \\
\hline
\end{array}
\]
\rowcolors{2}{white}{white}

\bigskip

In this table, each line corresponds to a projective representation, and we indicate which supersingular forms found in our search yield this representation. To each form, we associate as in section \ref{sect:XH} a modular curve $X_H(\ell N)$ whose jacobian contains the corresponding mod $\l$ linear representation, and we indicate the smallest of the genera of these curves\footnote{It is of course not impossible that the projective representation is afforded in a curve of smaller genus, but we do not know how to construct such a curve, nor if it exists.}.
Sadly, we have not found any example with a prime of degree higher than 1 satisfying the root discriminant bounds.

\bigskip

The first representation is the one attached to the newform
\[ \mfref{13.2.2.a} = q + (\zeta_3-1) q^2 - (2 \zeta_3-2) q^3 + O(q^4) \in \newf_2(13,\eps) \]
modulo any\footnote{Changing the prime is tantamount ot twisting the linear representation in this case.} of the two primes above 7 in $\Q(\zeta_3)$, where $\zeta_3$ is a primitive cube root of 1 and $\eps$ sends $2 \bmod 13$ to $-\zeta_3$. As the weight of this form is $2$, this representation occurs in the $13$-torsion of the genus 2 curve \href{http://www.lmfdb.org/Genus2Curve/Q/169/a/169/1}{$X_1(13)$}, so it is extremely easy to compute it with our algorithms. In fact, the corresponding degree 8 field is already part of \cite{KlunersMalle}.

On the other hand, as far as we know the second example is completely out of computational reach, which is a real pity as it is the first $\PSL_2$ example that we have encountered, and the value of $d_L$ is quite small.

\end{document}